\documentclass[a4paper,leqno]{amsart}
\usepackage{amsfonts, color, soul}
\usepackage[margin=3.5cm]{geometry}
\usepackage{soul}
\usepackage{changes}
\usepackage{tikz}

\theoremstyle{plain}
\begingroup
\newtheorem{theorem}{Theorem}[section]
\newtheorem{lemma}[theorem]{Lemma}
\newtheorem{proposition}[theorem]{Proposition}
\newtheorem{corollary}[theorem]{Corollary}
\endgroup

\theoremstyle{definition}
\begingroup
\newtheorem{definition}[theorem]{Definition}
\newtheorem{remark}[theorem]{Remark}

\endgroup

\theoremstyle{remark}

\newcommand{\res}{\mathop{\hbox{\vrule height 7pt width .5pt depth 0pt
\vrule height .5pt width 6pt depth 0pt}}\nolimits}

\newcommand{\R}{\mathbb{R}}
\newcommand{\N}{\mathbb{N}}
\newcommand{\LL}{\mathcal{L}}
\newcommand{\HH}{\mathcal{H}}
\newcommand{\Mn}{\mathbb{M}_{sym}^{n \times n}}
\newcommand{\MD}{\mathbb{M}^{n \times n}_D}
\newcommand{\wto}{\rightharpoonup}
\newcommand{\C}{\mathbf{C}}
\newcommand{\K}{\mathbf{K}}

\newcommand{\e}{\varepsilon}
\newcommand{\A}{\mathbf{A}}
\newcommand{\Acal}{\mathcal A}

\DeclareMathOperator{\Div}{div}
\DeclareMathOperator*{\esssup}{ess\, sup}

\newcommand{\Om}{\Omega}
\newcommand{\pscal}[2]{\langle #1, #2 \rangle}
\newcommand{\hs}{\mathcal{H}}

\definecolor{verde}{RGB}{20,150,100}

\mathchardef\emptyset="001F
\numberwithin{equation}{section}

\title[Spatial regularity for general yield criteria in  perfect plasticity]
{Spatial regularity for general yield criteria in dynamic and quasi-static perfect plasticity}
\author[J.-F. Babadjian, A. Giacomini and M.G. Mora]{Jean-Fran\c cois Babadjian, Alessandro Giacomini and Maria Giovanna Mora}

\address[J.-F. Babadjian]{Universit\'e Paris-Saclay, CNRS,  Laboratoire de math\'ematiques d'Orsay, 91405, Orsay, France.}
\email{jean-francois.babadjian@universite-paris-saclay.fr}

\address[A.~Giacomini]{DICATAM, Sezione di Matematica, Universit`a degli Studi di Brescia,
Via Branze 43, 25133 Brescia, Italy}
\email{alessandro.giacomini@unibs.it}

\address[M.G.~Mora]{Dipartimento di Matematica, Universit\`a di Pavia, 
Via Ferrata 5, 27100 Pavia, Italy}
\email{mariagiovanna.mora@unipv.it}


\begin{document}

\begin{abstract} 
This work addresses the question of regularity of solutions to evolutionary (quasi-static and dynamic) perfect plasticity models. Under the assumption that the elasticity set is a compact convex subset of  deviatoric matrices, with $C^2$ boundary and positive definite second fundamental form, it is proved that the Cauchy stress admits spatial partial derivatives that are locally square integrable. In the dynamic case, a similar regularity result is established for the velocity as well. In the latter case, one-dimensional counterexamples show that, although solutions are Sobolev in the interior of the domain, singularities may appear at the boundary and the Dirichlet condition may fail to be attained.
\end{abstract}
\maketitle

\tableofcontents

\section{Introduction}

Plasticity ($\pi\lambda\alpha\sigma\sigma\varepsilon\iota\nu$) is a typical inelastic behavior giving rise to permanent deformations inside materials. Beyond the reversible nature of elasticity, plasticity allows to mold, or give a shape, to a solid body. At the mesoscale, for polycrystalline solids, it turns out that deformations are highly localized due to slips along preferred directions. Those so-called plastic slips are due to the presence of defects (called dislocations) in the atomistic structure. From the macroscopic point of view, these singularities are due to a bad behavior of the dissipation energy (involving the plastic strain rate) which has a linear growth at infinity, and thus leads to possibly singular plastic deformations of measure type. One particular feature of plasticity models is therefore the creation of singularities due to concentration of deformations.\medskip

The model of small strain perfect plasticity can be described as follows (see, e.g., \cite{Lubliner}). Let $\Om \subset \R^n$ be a smooth bounded domain representing the reference configuration of an elasto-plastic body whose boundary $\partial \Om$ is partitioned into the disjoint union of a Dirichlet part $\Gamma_D$ and a Neumann part $\Gamma_N$. Let $f:\Om \times (0,T) \to \R^n$ be an external body force, $g:\Gamma_N \times (0,T) \to \R^n$ be an external traction force, and $w:\Gamma_D \times (0,T) \to \R^n$ be a  prescribed boundary displacement. A kinematically admissible displacement field $u:\Om \times (0,T) \to \R^n$ satisfies the Dirichlet boundary condition $u=w$ on $\Gamma_D \times (0,T)$, while a statically admissible Cauchy stress tensor $\sigma:\Om \times (0,T) \to \Mn$ satisfies the Neumann boundary condition $\sigma\nu=g$ on $\Gamma_N \times (0,T)$. In addition, in a dynamical framework, the following equation of motion holds:
\begin{equation}\label{eq:1}
\ddot u -\Div \sigma=f \quad \text{ in } \Om \times (0,T), 
\end{equation}
where $\ddot u$, the second partial derivative of $u$ with respect to time, represents the acceleration. In a quasi-static setting, the motion is assumed to be slow, so that the equation of motion is replaced by the equilbrium equation $-\Div \sigma=f$ in $\Om \times (0,T)$, where the acceleration is neglected. In small strain plasticity, the linearized strain tensor $Eu:=(Du+Du^T)/2$ additively decomposes as
\begin{equation}\label{eq:2}
Eu=e+p,
\end{equation}
where $e$ and $p:\Om \times (0,T) \to \Mn$ are the elastic ad plastic strains, respectively. Given a (symmetric and coercive) fourth order Hooke's tensor $\mathbf C$, the Hooke's law
\begin{equation}\label{eq:3}
\sigma=\mathbf C e
\end{equation}
gives the stress as a linear mapping of the elastic strain. In perfect plasticity, the Cauchy stress is constrained to stay in a fixed closed and convex set $\mathbf K \subset \Mn$. When $\sigma$ lies in the interior of $\mathbf K$, the material behaves elastically and no additional plastic strain is created ($\dot p=0$). On the other hand, when $\sigma$ reaches the boundary of $\mathbf K$, a plastic flow may develop in such a way that a non-trivial (permanent) plastic strain $p$ may remain after unloading. The evolution of $p$ is described by the  Prandtl-Reuss law $\dot p \in N_{\mathbf K}(\sigma)$, where $N_{\mathbf K}(\sigma)$ is the normal cone to $\mathbf K$ at $\sigma$, or equivalently by the Hill's principle of maximum plastic work
\begin{equation}\label{eq:4}
\sigma \cdot \dot p=H(\dot p),
\end{equation}
where $H(q):=\sup_{\tau \in \mathbf K}\tau\cdot q$ is the support function of $\mathbf K$. 

The problem of dynamic perfect plasticity thus consists in finding a quadruplet $(u,\sigma,e,p):\Om \times (0,T) \to \R^n \times \Mn \times \Mn \times \Mn$ satisfying the system \eqref{eq:1}--\eqref{eq:4}, supplemented by suitable initial conditions. 

It turns out that the following energy balance holds: for all $t \in [0,T]$,
\begin{multline*}
\frac12 \int_{\Om} |\dot u(t)|^2\, dx+\frac12 \int_\Om \mathbf C e(t)\cdot e(t)\, dx + \int_0^t\int_{\Om} H(\dot p)\, dx \, ds\\
= \frac12 \int_\Om|\dot u(0)|^2\, dx+ \frac12 \int_\Om \mathbf C e(0) \cdot e(0)\, dx\\
+\int_0^t \int_{\Om} f\cdot \dot u \, dx\, ds+ \int_0^t \int_{\Gamma_N}g\cdot \dot u\, d\HH^{n-1}\, ds + \int_0^t \int_{\Gamma_D}(\sigma\nu)\cdot \dot w\, d\HH^{n-1}\, ds.
\end{multline*}
It states that the variation of total energy (kinetic, elastic and cumulated dissipation energies) equals the work of external forces. Since $H(\dot p)\sim |\dot p|$, standard energy estimates lead to an $L^1$ bound on $\dot p$, which is thus expected to be a Radon measure. The natural energy space for a well-posed theory is therefore
$$(u(t),e(t),p(t)) \in BD(\Om) \times L^2(\Om;\Mn) \times \mathcal M(\Om \cup \Gamma_D;\Mn),$$
where $BD(\Om)$ is the space of functions of bounded variations introduced in \cite{Suquet2} for that purpose and
$\mathcal M(\Om \cup \Gamma_D;\Mn)$ is the space of $\Mn$-valued bounded Radon measures on $\Om \cup \Gamma_D$. Existence results have been obtained in \cite{A2,AG,T} for the static problem, in \cite{Suquet3,A,T,DMDSM,FG} in the quasi-static setting, and in \cite{AL,BMo2} in the dynamical case. The strategy of proof consists in a suitable regularization of the original model by means of visco-elastic (Kelvin-Voigt), or visco-plastic (Perzyna), or hardening models  and/or time discretization. Existence of solutions is then deduced by establishing uniform compactness estimates in the regularization parameter and then passing to the limit.\medskip

The purpose of this work is to investigate the regularity of solutions for the model of perfect plasticity, both in the quasi-static and in the dynamic settings, for smooth enough data. Up to now, most of the available results in this direction are limited to the case of a Von Mises yield criterion for which the elasticity set is explicitly given by
\begin{equation}\label{eq:VM}
\mathbf K=\{\sigma \in \Mn : \quad |\sigma_D| \leq 1\}.
\end{equation}
Here $\sigma_D:=\sigma-\frac{{\rm tr}(\sigma)}{n} {\rm Id}$ denotes the deviatoric part of $\sigma$. The first regularity result is the one in \cite{S1}, 
where by means of a duality method it was shown that the stress tensor $\sigma$  in the static case belongs to $H^1_{loc}(\Om;\Mn)$.
An extension to the quasi-static setting was then proved in \cite{BF} using a power law approximation (see also \cite{Dem}). Variants of this result have been obtained in \cite{Dav-Mor2} for perfect elasto-plastic plates and in \cite{BMo} for a pressure-dependent Drucker-Prager plasticity model. In dynamics, the only contributions discussing regularity are \cite{GMS} and \cite{Mi}. The paper  \cite{GMS} is an extension of \cite{Dav-Mor2} to dynamic perfect elasto-plastic plates, whereas \cite{Mi} is based on a completely different method (a comparison principle similar to the Kato inequality in the context of the wave equation) which allows to establish regularity in short time for compactly supported initial data (see also \cite{BMi}). In both articles $H^1_{loc}$-regularity in space is obtained not only for the Cauchy stress $\sigma(t)$, but also for the velocity $\dot u(t)$. However, contrary to \cite{S1,BF,Dem} which only hold for the Von Mises yield criterion \eqref{eq:VM}, the result of \cite{Mi} is valid for any convex set $\mathbf K$ (possibly depending on the hydrostatic pressure).

In the present article we prove Sobolev regularity results for the Cauchy stress (and also for the velocity in the dynamic case) for a wide class of convex elasticity regions of the form
$$\mathbf K=\{\sigma \in \Mn : \quad \sigma_D \in K\},$$
where $K$ is a compact and convex subset of $\MD$ (the space of symmetric $n \times n$ deviatoric matrices) with $C^2$ boundary and positive definite second fundamental form on $\partial K$, containing $0$ as an interior point. This type of yield criterion, which is invariant in the direction of hydrostatic matrices, applies to most of metals and alloys for which the influence of the mean stress on yielding is in general negligible. In such a case, the plastic strain $p$ becomes a deviatoric measure. Materials obeying this kind of law do not develop permanent volumetric changes and the displacement field admits only tangential discontinuities.

The case of the Von Mises yield criterion \eqref{eq:VM} appears as a particular case when $K$ is the closed unit ball of $\MD$. Our results also applies to anisotropic generalizations of the Von Mises criterion, due to Hill, where
$$K=\{\sigma \in \MD : \quad (\mathbf B \sigma)\cdot\sigma \leq 1\},$$
and $\mathbf B\in \mathcal L(\MD,\MD)$ is a fourth order self-adjoint and coercive tensor (see \cite{Lubliner}). Another example, which is covered by our result, is the Hosford criterion for which
$$
K=\bigg\{\sigma \in \MD : \quad \sum_{1 \leq i <j \leq n}|\sigma_i-\sigma_j|^p \leq 1\bigg\},
$$
where $\sigma_1,\ldots,\sigma_n$ are the eigenvalues of $\sigma$ and $p\geq 2$ (see Appendix, Proposition~\ref{prop:Hosford}). For $p=\infty$ the Hosford criterion reduces to the so-called Tresca criterion which instead does not fit within our framework. To our knowledge, it remains an open question to know whether the regularity of $\sigma$ (and $\dot u$ in the dynamic case) still holds in that case. However, as we mentioned earlier, at least short time regularity is ensured by the result in \cite{Mi}.\medskip

We now explain our proof strategy. For that, we simplify the model assuming the elasticity tensor $\mathbf C$ to be the identity, the body force $f$ to be zero, and the problem to be stationary. The system \eqref{eq:1}--\eqref{eq:4} thus formally simplifies to 
\begin{equation}\label{eq:minsigma}
\begin{cases}
v-\Div \sigma=0 & \text{ in }\Om,\\
Ev=\sigma+p & \text{ in }\Om,\\
p \in \partial I_{\mathbf K}(\sigma) & \text{ in }\Om,
\end{cases}
\end{equation}
where $I_{\mathbf K}$ is the indicator function of $\mathbf K$ ($I_{\mathbf K}=0$ in $\mathbf K$ and $I_{\mathbf K}=+\infty$ outside) and $\partial I_{\mathbf K}$ is the subdifferential of $I_{\mathbf K}$. This system is complemented by suitable boundary conditions that we do not consider here for simplicity. The idea of proof consists in approximating $I_{\mathbf K}$ by the power law function
$$\gamma_\alpha(\sigma):=\frac{\alpha}{\alpha+1}(1+d(\sigma)^2)^{\frac{\alpha+1}{2\alpha}},$$
where $0<\alpha \ll 1$ and $d$ denotes the distance function to $\mathbf K$ (see \cite{Tem}). The system \eqref{eq:minsigma} is then replaced by the so-called Norton-Hoff model
\begin{equation}\label{eq:minsigma2}
\begin{cases}
v_\alpha-\Div \sigma_\alpha=0 & \text{ in }\Om,\\
Ev_\alpha=\sigma_\alpha+D\gamma_\alpha(\sigma_\alpha) & \text{ in }\Om.
\end{cases}
\end{equation}
Elliptic regularity estimates show that solutions $(v_\alpha,\sigma_\alpha)$ exist and belong to $H^1(\Om;\R^n) \times H^1(\Om;\Mn)$. In order to control the spatial derivatives of $v_\alpha$ and $\sigma_\alpha$ uniformly with respect to $\alpha$, we formally take the partial derivative of the equations in \eqref{eq:minsigma2} with respect to the $x_k$ variable and test the first equation
with $\psi \partial_k v_\alpha$ and the second one with $\psi \partial_k \sigma_\alpha$ (here $\psi$ is a suitable cut-off function which is needed to avoid boundary terms). One is then faced with the application of the chain rule in Sobolev spaces to differentiate the nonlinear term $D\gamma_\alpha(\sigma_\alpha)$, which would require $D\gamma_\alpha$ to grow linearly with respect to its argument. Since this is not the case, we introduce a further approximation parameter $\lambda\gg 1$ and the regularized potential 
$$\gamma_{\alpha,\lambda}(\xi):= \frac{\alpha}{\alpha+1}(1+ (d^2(\xi){\,\wedge\,} \lambda^2))^{\frac{\alpha+1}{2\alpha}}
+  \frac{1}{2} (1+\lambda^2)^{\frac1{2\alpha}-\frac12} (d^2(\xi)  -\lambda^2)_+ .$$
We can now apply the previous program to the solution $(v_{\alpha,\lambda},\sigma_{\alpha,\lambda})$ of the system 
$$\begin{cases}
v_{\alpha,\lambda}-\Div \sigma_{\alpha,\lambda}=0 & \text{ in }\Om,\\
Ev_{\alpha,\lambda}=\sigma_{\alpha,\lambda}+D\gamma_{\alpha,\lambda}(\sigma_{\alpha,\lambda}) & \text{ in }\Om.
\end{cases}$$
Several integration by parts lead to the main estimate \eqref{eq:estim2}, where the obstruction to conclude the argument is the presence of a term of the form 
$$\int_\Om |Ev_{\alpha,\lambda}| |\nabla (\sigma_{\alpha,\lambda})_D|\psi\, dx$$
on the right-hand side. In fact, $Ev_{\alpha,\lambda}$ is only uniformly bounded in $L^1(\Om;\Mn)$. 
We show (see \eqref{eq:ref:intro}) that the integral above can be absorbed by the coercive term
$$\int_\Om \psi \partial_k D\gamma_{\alpha,\lambda}(\sigma_{\alpha,\lambda})\cdot\partial_k \sigma_{\alpha,\lambda} \, dx$$
appearing in the left-hand side of \eqref{eq:estim2}. It is at this stage that the regularity of $\partial K$ and the sign condition on its second fundamental form are used, providing a uniform bound (with respect to the parameters $\alpha$ and $\lambda$) of $\|\nabla v_{\alpha,\lambda}\|_{L^2_{loc}(\Om)}$ and $\|\nabla \sigma_{\alpha,\lambda}\|_{L^2_{loc}(\Om)}$. Passing to the limit as $\alpha \to 0$ and $\lambda \to +\infty$ finally yields the desired $H^1_{loc}$ estimate of $v$ and $\sigma$.

This general strategy is adapted to the dynamic case leading to Sobolev regular solutions to the system \eqref{eq:1}--\eqref{eq:4}. In particular, the flow rule \eqref{eq:4}, which is usually formulated in a measure theoretical sense due to the lack of regularity of the solutions in the energy space (see \cite{KT,FG}), can now be expressed in a pointwise almost every sense. In the quasi-static case this method gives $H^1_{loc}$ regularity for the stress only (no regularity is obtained for the velocity). Using a capacitary argument and the $H^1$-quasi-continuous representative of the stress, this result allows one to give a pointwise form to the flow rule, as previously observed in \cite{FGM}.\medskip

The manuscript is organised as follows. Section~\ref{sec:prel} is devoted to notation as well as mathematical preliminary results used throughout the paper (in particular, functional spaces and basic properties of the distance function to a convex set). In Section~\ref{sec:setting} we introduce the dynamic perfect plasticity model as well as its Norton-Hoff approximation. We state and prove our first main regularity result, Theorem~\ref{prop:reg-sN2}, and derive a well-posedness result, Theorem~\ref{thm:hreg2}, of strong solutions. Section~\ref{sec:qst} concerns the quasi-static perfect plasticity model. Our second main regularity result is stated in Theorem~\ref{prop:reg-sN}, from which we deduce an existence result for quasi-static evolutions with a locally Sobolev regular Cauchy stress, Theorem~\ref{thm:hreg3}. In Section~\ref{sec:ex} we present two one-dimensional examples of dynamic solutions (in the stationary and non-stationary cases) that are smooth in the interior of the domain, but develop discontinuities at the boundary. Finally, in the Appendix we collect some supplemental material. In particular, we prove a general well-posedness result that avoids the usual time discretization method and uses instead an ODE argument based on the Cauchy-Lipschitz Theorem. Moreover, we show that the Hosford criterion for $p\geq2$ satisfies all the assumptions of our regularity results.

\section{Preliminaries}
\label{sec:prel}

\subsection{Notation}

\subsubsection{Linear algebra}
	
If $a$ and $b \in \R^n$, we write $a \cdot b:=\sum_{i=1}^n a_i b_i$ for the Euclidean scalar product, and we denote by $|a|:=\sqrt{a \cdot a}$ the corresponding norm. 
	
We denote by $\mathbb M^{n\times n}$ the set of $n \times n$ matrices and by $\Mn$ the space of symmetric $n \times n$ matrices. The set of all (deviatoric) trace free symmetric matrices will be denoted by $\MD$. The space $\mathbb M^{n\times n}$ is endowed with the Frobenius scalar product $A\cdot B:={\rm tr}(A^T B)$ and with the corresponding Frobenius norm $|A|:=\sqrt{A\cdot A}$. If $a \in \R^n$ and $b \in \R^n$, we denote by $a \odot b:=(ab^T+b^T a)/2 \in  \Mn$ their symmetric tensor product.
	
If $A \in \Mn$, there exists an orthogonal decomposition of $A$ with respect to the Frobenius scalar product as follows
$$ A = A_D + \frac{1}{n} ({\rm tr} A){\rm Id} ,$$
where $A_D \in \MD$ stands for the deviatoric part of $A$. 
	
\subsubsection{Measures}

The Lebesgue measure in $\R^n$ is denoted by $\mathcal L^n$ and the $(n-1)$-dimensional Hausdorff measure by $\mathcal H^{n-1}$. If $X \subset \R^n$ is a locally compact Borel set and $Y$ is an Euclidean space, we denote by $\mathcal M(X;Y)$ the space of $Y$-valued bounded Radon measures in $X$ endowed with the norm $\|\mu\|:=|\mu|(X)$, where $|\mu|$ is the variation of the measure $\mu$. If $Y=\R$ we simply write $\mathcal M(X)$ instead of $\mathcal M(X;\R)$. If $\mu=f \LL^n$ for some (locally integrable) function $f$, we will often identify the measure $\mu$ to its density $f$.

By the Riesz representation theorem, $\mathcal M(X;Y)$ can be identified with the dual space of $C_0(X;Y)$, the space of continuous functions $\varphi:X \to Y$ such that $\{|\varphi|\geq \varepsilon\}$ is compact for every $\varepsilon>0$. The (vague) weak* topology of $\mathcal M(X;Y)$ is defined using this duality. 
	
Let $\mu \in \mathcal M(X;Y)$ and $h:Y \to [0,+\infty]$ be a convex, positively one-homogeneous function. Using the theory of convex functions of measures developed in \cite{DT1,GS}, we introduce the nonnegative {Borel measure $h(\mu)$}, defined by 
$$h(\mu)=h\left(\frac{\mu}{|\mu|}\right)|\mu|\,,$$
where $\frac{\mu}{|\mu|}$ stands for the Radon-Nikod\'ym derivative of $\mu$ with respect to $|\mu|$.
	
\subsubsection{Functional spaces}
	
We use standard notation for Lebesgue spaces ($L^p$) and Sobolev spaces ($W^{s,p}$ and $H^s=W^{s,2}$).
	
Let $\Omega\subset\R^n$ be an open set. The space of functions of bounded deformation is defined by
$$BD(\Omega)=\{u \in L^1(\Omega;\R^n) : \; E u \in \mathcal M(\Omega;\Mn)\}\,,$$
where $E u:=(Du+Du^T)/2$ stands for the distributional symmetric gradient of $u$.  We recall (see \cite{Bab,T}) that, if $\Omega$ has a Lipschitz boundary, every function $u \in BD(\Omega)$ admits a trace, still denoted by $u$, which belongs to $L^1(\partial\Omega;\R^n)$, and such that the integration by parts formula holds: for all $\varphi \in C^1(\overline \Omega;\Mn)$,
$$\int_{\partial\Omega} u\cdot (\varphi\nu)\, d\mathcal H^{n-1}=\int_\Omega {\rm div} \varphi \cdot u\, dx + \int_\Omega \varphi \cdot d E u,$$
where $\nu$ is the outer unit normal to $\partial \Omega$.
	
Let us define 
$$H_{\Div}(\Omega;\Mn)=\{\sigma \in L^2(\Omega;\Mn) :\;  {\rm div} \sigma \in L^2(\Omega;\R^n)\}\,.$$
If $\Omega$ has Lipschitz boundary, for any $\sigma \in H_{\Div}(\Omega;\Mn)$ we can define the normal trace $\sigma\nu$ as an element of $H^{-\frac12}(\partial \Omega;\R^n)$  (cf., e.g.,\ \cite[Theorem~1.2, Chapter~1]{T}) by setting
$$\langle \sigma \nu, \psi \rangle_{H^{-\frac12}(\partial \Omega;\R^n),H^{\frac12}(\partial \Omega;\R^n)}:= \int_\Omega \psi \cdot {\rm div} \sigma \, dx + \int_\Omega \sigma \cdot  E \psi \, dx$$
for every $\psi \in H^1(\Omega;\R^n)$. 

\subsection{Banach space-valued functions}

Let $X$ be a Banach space and let $1 \leq p \leq \infty$. If $X$ is a reflexive (respectively, the dual of a separable) space, we define $L^p(0,T;X)$ as the space of all strongly (respectively, weakly*) measurable functions $f:(0,T) \to X$ such that $t \mapsto \|f(t)\|_X$ belongs to $L^p(0,T)$. According to the Riesz representation Theorem, if $1 \leq p <\infty$ and $X$ is reflexive, the dual space of $L^p(0,T;X)$ is isometrically isomorphic to $L^q(0,T;X')$ with $\frac1p+\frac1q=1$. If $X$ is the dual of a separable Banach space $Y$ and $1 \leq p <\infty$, the dual space of $L^p(0,T;Y)$ is isometrically isomorphic to $L^q(0,T;X)$ with $\frac1p+\frac1q=1$. In particular, $L^q(0,T;X)$ can be naturally endowed with a weak* topology. This result will be in particular applied to $X=BD(\Omega)$, which can be identified with the dual of a separable space (see \cite[Proposition~2.5]{ST} and also \cite[Remark~3.12]{AFP} in the $BV$ case). We refer to \cite[Section~2.3]{FL} for a general theory of Lebesgue spaces on Banach spaces. 

We further denote by $AC([0,T];X)$ the space of absolutely continuous functions $f : [0,T] \to X$. 
Let $X$ be a reflexive (respectively, the dual of a separable) space. By \cite[Appendix]{Br} (respectively, \cite[Theorem 7.1]{DMDSM}),
if $f \in AC([0,T];X)$, then the time derivative $\dot f$ exists $\mathcal L^1$-a.e.\ in $(0,T)$ with respect to the strong (respectively, weak*) convergence in $X$ and $\dot f \in L^1(0,T;X)$. We further define the Sobolev space 
$$W^{1,p}(0,T;X)=\big\{f \in AC([0,T];X) : \; \dot f \in L^p(0,T;X)\big\}.$$
We will use the notation $H^1(0,T;X)$ to denote $W^{1,2}(0,T;X)$.

\subsection{Spaces of admissible fields and duality between stress and plastic strain}

Let $\Omega \subset \R^n$ be a connected bounded open set. We consider an open subset $\Gamma_D \subset \partial\Omega$  with respect to the relative topology of $\partial \Omega$ which stands for the Dirichlet part of $\partial \Omega$. We also define $\Gamma_N=\partial\Omega \setminus \overline \Gamma_D$ which corresponds to the Neumann part of the boundary.

Given a prescribed boundary displacement $w \in H^1(\Om;\R^n)$, we will consider the following space of kinematically admissible fields. In the dynamic case we set 
\begin{multline*}
\mathcal A^{\rm dyn}_w:=\Big\{(v,\eta,q) \in [BD(\Om)\cap L^2(\Om;\R^n)] \times L^2(\Om;\Mn) \times \mathcal M(\Om \cup \Gamma_D;\MD) :\\
Ev=\eta+q \text{ in }\Om, \quad q=(w-v) \odot \nu  \HH^{n-1} \text{ on }\Gamma_D\Big\},
\end{multline*}
where $\nu$ is the outer unit normal to $\partial\Om$. In the quasistatic case the absence of the kinetic energy prevents an $L^2$ control on the displacement (or velocity), so that the previous space is replaced by
\begin{multline*}
\mathcal A^{\rm qst}_w:=\Big\{(v,\eta,q) \in BD(\Om) \times L^2(\Om;\Mn) \times \mathcal M(\Om \cup \Gamma_D;\MD) :\\
Ev=\eta+q \text{ in }\Om, \quad q=(w-v) \odot \nu  \HH^{n-1} \text{ on }\Gamma_D\Big\}.
\end{multline*}

Given a surface traction $g \in L^\infty(\Gamma_N;\R^n)$, the space of statically admissible stresses in the dynamical case is defined by
$$\mathcal S^{\rm dyn}_g :=\{ \tau \in H_{\Div}(\Om;\Mn) : \; \tau_D \in L^\infty(\Om;\MD),\, \tau\nu=g \text{ on }\Gamma_N\},$$
while in the quasistatic case
$$\mathcal S^{\rm qst}_g :=\{ \tau \in L^2(\Om;\Mn) : \; \Div \tau \in L^n(\Om;\R^n), \, \tau_D \in L^\infty(\Om;\MD),\, \tau\nu=g \text{ on }\Gamma_N\}.$$
Note that, if $\Om$ is a $C^2$ domain, then, for any $\tau \in \mathcal S^{\rm qst}_g$ (resp. $\mathcal S^{\rm dyn}_g$), the tangential part of the normal stress $(\sigma\nu)_\tau:=\sigma\nu - \langle\sigma\nu,\nu\rangle \nu \in L^\infty(\partial\Om;\R^n)$ (see \cite[Lemma 2.4]{KT} or \cite[Section~1.2]{FG}).

The duality pairing between stresses and plastic strains is {\em a priori} not well defined, since the former are only squared Lebesgue integrable, while the latter are (possibly singular) measures. Using an integration by parts formula as in \cite{KT,FG}, it is possible to give a meaning to this duality pairing in terms of a distribution, and even of a measure.

\begin{definition}\label{def:duality}
Let $w \in H^1(\Om;\R^n)$ and $g \in L^\infty(\Gamma_N;\R^n)$. For all $(u,e,p) \in \mathcal A^{\rm dyn}_w$ (resp. $ \mathcal A^{\rm qst}_w$) and all $\sigma \in \mathcal S^{\rm dyn}_g$  (resp. $ \mathcal S^{\rm qst}_g$), we define the distribution $[\sigma_D \cdot p]$ on $\R^n$ as
\begin{multline}
\label{eq:duality-distr}
\langle [\sigma_D\cdot p], \varphi\rangle = \int_\Om \varphi(w-u)\cdot\Div\sigma\, dx +\int_\Om \sigma\cdot [(w-u)\odot\nabla\varphi]\, dx\\
+\int_\Om \sigma\cdot (Ew-e)\varphi\, dx+\int_{\Gamma_N}\varphi g \cdot (u-w) \, d \HH^{n-1}
\end{multline}
for every $\varphi\in C^\infty_c(\R^n)$. 
\end{definition}

According to \cite[Theorems~6.3 \& 6.5]{FG}, $[\sigma_D\cdot p]$ extends to a bounded Radon measure in $\R^n$ with\footnote{The proof of \cite{FG} actually requires that $\Div\sigma \in L^n(\Om;\R^n)$. In the dynamic case, the same conclusion holds if we only suppose $\Div \sigma \in L^2(\Om;\R^n)$ and $u \in L^2(\Om;\R^n)$.}
$$|[\sigma_D\cdot p]| \leq \|\sigma_D\|_{L^\infty(\Om)} |p| \quad \text{ in }\mathcal M(\R^n)$$
and
\begin{equation}\label{intbyp}
\langle \sigma_D,p\rangle:=[\sigma_D\cdot p](\R^n)=\int_\Om (w-u)\cdot\Div\sigma\, dx+\int_\Om \sigma\cdot (Ew-e)\, dx+\int_{\Gamma_N} g \cdot (u-w) \, d \HH^{n-1}.
\end{equation}

\subsection{Some results concerning the distance function from a convex set} 

Let $K\subseteq \R^N$ be a closed convex set such that
\begin{equation}
\label{eq:Kballs}
B(0,r_K)\subset K\subset B(0,R_K)
\end{equation}
for some $r_K,R_K>0$. Let $d_K:\R^N \to [0,+\infty)$ and $\Pi_K:\R^N\to K$ be the distance function to $K$ and the projection onto $K$, respectively.

\begin{theorem}
\label{thm:distance}
 Let $K\subseteq \R^N$ be a closed convex set satisfying \eqref{eq:Kballs}. The following properties hold true.
\begin{itemize}
\item[(a)] The function $d_K^2$ is differentiable on $\R^N$ and for every $x\in\R^N$
$$
\nabla d_K^2(x)=2(x-\Pi_K(x)).
$$
\item[(b)] For every $x\in\R^n$
$$
(x-\Pi_K(x))\cdot x\geq d_K^2(x)\qquad\text{and}\qquad (x-\Pi_K(x))\cdot x\geq r_Kd_K(x). 
$$
\item[(c)] If $\partial K$ is of class $C^2$ and its second fundamental form is positive definite at every point of $\partial K$, then $\Pi_K$ is of class $C^1$ in $\R^N \setminus K$ and there exists $C_K>0$ such that
$$
D\Pi_K(x)v\cdot v \leq  \frac1{1+C_K d_K(x)} |v|^2 
$$
for every $x\in \R^N\setminus K$ and $v\in\R^N$.
\end{itemize}
\end{theorem}

\begin{proof}
Point (a) is standard (see, e.g., \cite[pp. 286]{Moreau}). Concerning point (b), we have
$$
(x-\Pi_K(x))\cdot x=|x-\Pi_K(x)|^2+(x-\Pi_K(x))\cdot \Pi_K(x)\ge |x-\Pi_K(x)|^2=d_K^2(x)
$$
since in view of the projection property, as $0\in K$,
$$
-(x-\Pi_K(x))\cdot \Pi_K(x)=(x-\Pi_K(x))\cdot (0-\Pi_K(x))\le 0.
$$
Concerning the second inequality, we may assume $x\not\in K$. Then by \eqref{eq:Kballs}
$$
r_K\frac{x-\Pi_K(x)}{|x-\Pi_K(x)|} \in K,
$$
and thus the properties of the projection imply that
$$
0\leq (x-\Pi_K(x))\cdot\Big(x-r_K \frac{x-\Pi_K(x)}{|x-\Pi_K(x)|}\Big)\\
=(x-\Pi_K(x))\cdot x-r_K  d_K(x),
$$
so that the second inequality follows.
\par
Let us come to point (c).
If  the boundary of $K$ is of class $C^2$, it turns out that the projection $\Pi_K$ is of class $C^1$ in $\R^N\setminus K$ with
$$
D\Pi_K(x) = \big(I +d_K(x) A(\Pi_K(x))\big)^{-1}\big({\rm Id}-\nu_K(\Pi_K(x))\otimes\nu_K(\Pi_K(x))\big)
$$
for every $x\in\R^N\setminus K$, where $\nu_K$ is the outer unit normal to $\partial K$ and $A=D\nu_K$ is the second fundamental form of $\partial K$ (see, e.g., \cite{Hol}). If $A$ is positive definite, we get
\begin{equation}\label{assK}
A(x)v \cdot v \geq C_K |v |^2   \quad \text{ for every $x\in \partial K$ and $v\in T_x(\partial K),$}
\end{equation}
where $T_x(\partial K)$ is the tangent space to $\partial K$ at $x$. Now let us fix $x\in\R^N\setminus K$ and for simplicity of notation let us set
$$
S:=d_K(x) A(\Pi_K(x)) \qquad\text{and}\qquad Q:={\rm Id}-\nu_K(\Pi_K(x))\otimes\nu_K(\Pi_K(x)).
$$ 
Note that $Q$ is the orthogonal projection onto the tangent space to $\partial K$ at $\Pi_K (x)$. For every $v\in\R^N$ we deduce that
\begin{equation}\label{part1}
D\Pi_K(x)v\cdot v=({\rm Id}+S)^{-1}Qv\cdot v\leq |({\rm Id}+S)^{-1}Qv|\, |v|.
\end{equation}
Setting $\zeta:=({\rm Id}+S)^{-1}Qv$ and using \eqref{assK}, we have
$$
|Qv|=|({\rm Id}+S)\zeta|\geq (1+C_K d_K(x))|\zeta|.
$$
This can be rewritten as
$$
|({\rm Id}+S)^{-1}Qv|\leq \frac1{1+C_K d_K(x)} |Qv|, 
$$
which, together with \eqref{part1}, implies the inequality of point (c).
\end{proof}

By composition, we easily infer that the following result holds true.

\begin{corollary}
\label{cor:ineq}
Let $K\subseteq \R^N$ be a closed convex set satisfying \eqref{eq:Kballs}, $\Phi:\R\to \R$ be an increasing function, and let us set $\Psi(x):=\Phi(d_K^2(x))$. 
The following properties hold true.
\begin{itemize}
\item[(a)] If $\Phi$ is of class $C^1$, then for every $x\in \R^N$
$$
\nabla \Psi(x)=2\Phi'(d_K^2(x))(x-\Pi_K(x))
$$
and 
$$
\nabla \Psi(x)\cdot x\ge 2\Phi'(d_K^2(x)) d_K^2(x)\qquad\text{and}\qquad
\nabla \Psi(x)\cdot x \ge 2\Phi'(d_K^2(x)) r_K d_K(x).
$$
\item[(b)] If $\partial K$ is of class $C^2$ and $\Phi$ is of class $C^2$, then for every $x\in\R^N\setminus K$ and $v\in\R^N$
$$
D^2\Psi(x)v\cdot v=4\Phi''(d_K^2(x))[(x-\Pi_K(x))\cdot v]^2+2\Phi'(d_K^2(x))(v-D \Pi_K(x)v)\cdot v.
$$
If in addition the second fundamental form of $\partial K$ is positive definite at every point of $\partial K$, then there exists $C_K>0$ such that for every $x\in \R^N\setminus K$ and $v\in\R^N$
$$
D^2\Psi(x)v\cdot v\ge 4\Phi''(d_K^2(x))[(x-\Pi_K(x))\cdot v]^2+2\Phi'(d_K^2(x))\frac{C_Kd_K(x)}{1+C_Kd_K(x)}|v|^2.
$$
\end{itemize}
\end{corollary}

\section{The dynamical model}\label{sec:setting}

\subsection{Setting of the problem}

In the following we describe the setting of the model of perfect elasto-plasticity.
\vskip10pt\par\noindent{\bf $(H_1)$ Elasticity tensor.}
Let $\C$ be the elasticity tensor, considered as a symmetric positive definite linear operator $\C: \Mn\to\Mn$, 
and let $\A: \Mn\to\Mn$ be its inverse $\A:=\C^{-1}$. It follows that there exist two positive constants $\alpha_\A$ and $\beta_\A$, with $\alpha_\A \leq \beta_\A$, such that
\begin{equation}\label{coercA}
\alpha_\A |\xi|^2 \leq \frac12  (\A\xi)\cdot\xi \leq \beta_\A |\xi|^2 \quad \text{ for every } \xi \in \Mn.
\end{equation}

\vskip10pt
\noindent{\bf $(H_2)$ Reference configuration.} Let $\Om\subset \R^n$ be a bounded open set with Lipschitz boundary. We will consider a decomposition of the boundary in the form
$$
\partial\Om:=\Gamma_D \cup \Gamma_N\cup \Gamma,
$$
where $\Gamma_D$ and $\Gamma_N$ are relatively open with relative boundary given by $\Gamma$. We assume that there exists an open set $V$ containing $\Gamma$ such that $V \cap \partial \Om$ is an $(n-1)$-dimensional submanifold of $\R^n$ of class $C^2$, and that $\Gamma$ itself is an $(n-2)$-dimensional $C^2$ submanifold of $\R^n$.

\vskip10pt
\noindent{\bf $(H_3)$ Yield region.} Let $\K$ be a closed convex set in $\Mn$ of the form 
\begin{equation}
\label{eq:defK}
\K:=K+ \R\, {\rm Id},
\end{equation}
where $K$ is a closed convex subset of $\MD$. We assume that there exist two positive constants $r_K$ and $R_K$, with $r_K \leq R_K$, such that
\begin{equation} \label{boundK}
B( 0, r_K) \subset K \subset B( 0, R_K).
\end{equation}
We denote by $H:\MD \to [0,+\infty)$ the support function of $K \subset \MD$  defined by
$$H(q):=\sup_{\tau \in K} \tau \cdot q,$$
which is a convex and positively one-homogeneous function satisfying by \eqref{boundK}
$$r_K |q| \leq H(q) \leq R_K |q| \quad \text{ for all }q \in \MD.$$
\vskip10pt
\noindent{\bf $(H_4)$ Prescribed displacements.} We will consider a time-dependent boundary displacement $t \mapsto w(t)$ which is the trace on $\Gamma_D$ of a function $w(t)$ with
\begin{equation}
\label{eq:w}
w\in H^2(0,T;H^1(\Omega;\R^n)) \cap H^3(0,T;L^2(\Omega;\R^n)).
\end{equation}
\vskip10pt
\noindent{\bf $(H_5)$ External loads.} We consider time-dependent external loads made of body forces $t \mapsto f(t)$ and traction forces $t \mapsto g(t)$ applied on $\Gamma_N$ with
$$
f \in W^{1,\infty}(0,T;L^2(\Om;\R^n))\qquad\text{and}\qquad g\in W^{1,\infty}(0,T;L^\infty(\Gamma_N;\R^n))
$$
For every $v \in H^1_{\Gamma_D}(\Om;\R^n)$ (the subspace of $H^1(\Om;\R^n)$ made of all function $v \in H^1(\Om;\R^n)$ such that $v=0$ on $\Gamma_D$ in the sense of traces), we define
$$
\pscal{\LL(t)}{v}:=\int_\Om f(t)\cdot v\,dx+\int_{\Gamma_N}g(t)\cdot v\,d\hs^{n-1}.
$$
We assume that the loads can be represented by a potential
\begin{equation}
\label{eq:rhot}
\rho\in H^2(0,T;L^2(\Om ;\Mn))
\end{equation}
with
\begin{equation}
\label{eq:rhotD}
\rho_D\in W^{2,\infty}(0,T;L^\infty(\Om ;\MD)),
\end{equation}
in the following sense: 
\begin{equation}
\label{eq:LL}
\pscal{\LL(t)}{v}=\int_\Om \rho(t)\cdot Ev\,dx \quad \text{ for every }v\in H^1_{\Gamma_D}(\Om;\R^n),
\end{equation}
i.e., $\rho\in W^{1,\infty}(0,T;H_{\Div}(\Om;\Mn))$ and satisfies
$$
\begin{cases}
-{\rm div}\rho(t)=f(t)&\text{ in $L^2(\Om;\R^n)$,}\\
\rho(t)\nu=g(t) &\text{ in }L^\infty(\Gamma_N;\R^n).
\end{cases}
$$
We further assume that it satisfies the uniform safe load condition: for all $t \in [0,T]$
\begin{equation}
\label{eq:rhoK}
r_K-\|\rho_D(t)\|_{L^\infty(\Om)} \ge c>0,
\end{equation}
where $r_K$ is the constant appearing in \eqref{boundK}.

\vskip10pt

\noindent{\bf $(H_6)$ Initial conditions.} 
Let $\sigma_0$ be an initial stress such that
\begin{equation}
\label{eq:sigma0S1}
\sigma_0\in H_{\Div}(\Om;\Mn),\qquad \sigma_0(x)\in \K \text{ for a.e. $x\in \Om$,}
\end{equation}
and
\begin{equation}
\label{eq:sigma0S2}
\begin{cases}
-\Div \sigma_0=f(0)&\text{in }\Om,\\
\sigma_0\nu=g(0)&\text{on }\Gamma_N.
\end{cases}
\end{equation}
In the dynamical case, we also consider an initial velocity $v_0\in H^1(\Omega;\R^n)$ satisfying $v_0=\dot w(0)$ on~$\Gamma_D$. Note that this last condition is irrelevant in the quasistatic case.

\par

\subsection{Norton-Hoff in dynamical perfect plasticity}
Let $\K=K+ \R\, {\rm Id}$ be the yield region defined in \eqref{eq:defK}. For $\xi\in\Mn$ we define 
$$
d(\xi):=d_\K(\xi)=d_K(\xi_D),
$$
where $\xi=\xi_D +\frac{{\rm tr}(\xi)}{n}{\rm Id}$ is the orthogonal decomposition of $\xi$ with $\xi_D \in \MD$. 
We have clearly
$$|\xi_D|-R_K\le d(\xi)\le |\xi_D|.$$
Given $\alpha\in(0,1]$ and $\lambda>0$, we consider the $C^1$ function on $\gamma_{\alpha,\lambda}:\Mn \to \R$ given by
$$\gamma_{\alpha,\lambda}(\xi):= \frac{\alpha}{\alpha+1}(1+ (d^2(\xi){\,\wedge\,} \lambda^2))^{\frac{\alpha+1}{2\alpha}}
+  \frac{1}{2} (1+\lambda^2)^{\frac1{2\alpha}-\frac12} (d^2(\xi)  -\lambda^2)_+ .$$
Note that $\gamma_{\alpha,\lambda}$ is a convex function, since it is an increasing and convex function of $d^2(\xi)$.
We collect some useful inequalities involving the derivatives of $\gamma_{\alpha,\lambda}$ that will be employed in several estimates and are consequences of Theorem~\ref{thm:distance} and Corollary \ref{cor:ineq}.
\begin{itemize}
\item[(a)] The function $\gamma_{\alpha,\lambda}$ is continuously differentiable in $\Mn$ with gradient given by
\begin{equation}\label{Dgamma}
D\gamma_{\alpha,\lambda}(\xi)=
(1+ (d^2(\xi){\,\wedge\,} \lambda^2))^{\frac1{2\alpha}-\frac12}(\xi-\Pi_\K(\xi)).
\end{equation}
Since $\K$ is invariant in the direction of hydrostatic matrices ($\R\,  {\rm Id}$), it follows that $D\gamma_{\alpha,\lambda}(\xi) \in \MD$. Moreover, for every $\xi\in \Mn$ we have the inequalities
\begin{equation}
\label{eq:Dg-1}
D\gamma_{\alpha,\lambda}(\xi)\cdot \xi \geq  (1+d^2(\xi)\wedge \lambda^2)^{\frac1{2\alpha}-\frac12}d^2(\xi)  
\end{equation}
and
\begin{equation}\label{Dg-2}
D\gamma_{\alpha,\lambda}(\xi)\cdot \xi \geq r_K(1+d^2(\xi)\wedge \lambda^2)^{\frac1{2\alpha}-\frac12}d(\xi)=r_K|D\gamma_{\alpha,\lambda}(\xi)|  .
\end{equation}

\item[(b)] Let us assume in addition that $\partial K$ is of class $C^2$ and that its second fundamental form is positive definite at every point. 
In view of the cylindrical geometry of $\K$, we have that for all $\xi \in \Mn$
$$
\Pi_\K(\xi)=\Pi_K(\xi_D)+\frac{{\rm tr}(\xi)}{n}{\rm Id}.
$$
In particular, if $\eta \in \Mn$ and $\xi\in \Mn\setminus \K$
$$
D\Pi_\K(\xi)\eta=D\Pi_K(\xi_D)\eta_D+\frac{{\rm tr}(\eta)}{n}{\rm Id},
$$
so that
\begin{equation}
\label{eq:ineqprojK}
D\Pi_\K(\xi)\eta\cdot \eta\le \frac{1}{1+C_Kd(\xi)}|\eta_D|^2+\frac{1}{n}({\rm tr}(\eta))^2.
\end{equation}

\item[(c)] Let $\sigma\in H^1_{loc}(\Om;\Mn)$. Since by \eqref{Dgamma} the function $D\gamma_{\alpha,\lambda}$ is Lipschitz continuous and piecewise $C^1$, it follows from the generalized chain rule formula in Sobolev spaces \cite[Proposition 1.2]{MT} that $D\gamma_{\alpha,\lambda}(\sigma) \in H^1_{loc}(\Om;\Mn)$. Moreover, by point (b) in Corollary~\ref{cor:ineq} we conclude that for every $k=1,\dots,n$
\begin{equation}\label{eq:formD2gamma2}
\partial_k \big(D\gamma_{\alpha,\lambda}(\sigma)\big)\cdot \partial_k\sigma\geq  
(1+d^2(\sigma)\wedge \lambda^2)^{\frac{1}{2\alpha}-\frac12}\frac{C_Kd(\sigma)|\partial_k\sigma_D|^2}{1+C_K d(\sigma)} \qquad\text{a.e.\ in $\Om$.}
\end{equation}
\end{itemize}

In the following dynamic evolution problems, we set the mass density of the material to be identically equal to $1$. In the next theorem we consider a so-called ``Norton-Hoff'' approximation model and we prove its well-posedness and some regularity results.

\begin{theorem}[\bf Norton-Hoff approximation]
\label{thmNH-dyn}
Assume hypotheses $(H_1)$--$(H_6)$.  Then for every $\alpha\in(0,1]$ and $\lambda>0$, there exists a unique pair $(\sigma^{\alpha,\lambda}, v^{\alpha,\lambda})$ with
\begin{equation}\label{eq:sigmaal}
\sigma^{\alpha,\lambda} \in W^{1,\infty}(0,T;L^2(\Omega;\Mn)) \cap L^\infty(0,T;H_{\Div}(\Omega;\Mn)),
\end{equation}
and
\begin{equation}\label{eq:val}
v^{\alpha,\lambda} \in W^{1,\infty}(0,T;L^2(\Omega;\R^n)) \cap L^\infty(0,T;H^1(\Omega;\R^n)), 
\end{equation}
that solves the system
\begin{equation}\label{NH-pb2}
\begin{cases}
\A\dot\sigma^{\alpha,\lambda}+D\gamma_{\alpha,\lambda}(\sigma^{\alpha,\lambda})=Ev^{\alpha,\lambda} & \text{ in }\Omega \times (0,T),
\\
\dot v^{\alpha,\lambda} -\Div\sigma^{\alpha,\lambda}=f  & \text{ in }\Omega \times (0,T),
\\
v^{\alpha,\lambda}=\dot w & \text{ on }\Gamma_D \times (0,T),
\\
\sigma^{\alpha,\lambda}\nu=g & \text{ on }\Gamma_N \times (0,T),
\\
(\sigma^{\alpha,\lambda}(0),v^{\alpha,\lambda}(0))=(\sigma_0,v_0) & \text{ in }\Omega.
\end{cases}
\end{equation}
Moreover the following estimates hold:
\begin{equation}\label{est1-sN2}
\|\sigma^{\alpha,\lambda}\|_{W^{1,\infty}(0,T;L^2(\Omega))}\leq C, \qquad  \|v^{\alpha,\lambda}\|_{W^{1,\infty}(0,T;L^2(\Omega))}\leq C,
\end{equation}
\begin{equation}\label{est1-vN2}
\|v^{\alpha,\lambda}\|_{L^\infty(0,T;BD(\Omega))} \leq C,
\end{equation}
\begin{equation}\label{est2-sN2}
\sup_{t\in[0,T]}\int_\Omega \big(1+d^2(\sigma^{\alpha,\lambda}(t))\wedge \lambda^2)\big)^{\frac1{2\alpha}-\frac12} d^\ell(\sigma^{\alpha,\lambda}(t))\, dx \leq C \quad \text{ for } \ell=1,2,
\end{equation}
\begin{equation}\label{eq:energy-bound}
\sup_{t \in [0,T]}\int_\Om \gamma_{\alpha,\lambda}(\sigma^{\alpha,\lambda}(t))\, dx \leq C,
\end{equation}
and
\begin{equation}\label{est2-uni2}
\sup_{t\in[0,T]} \int_\Omega \big(1+d^2(\sigma^{\alpha,\lambda}(t))\wedge \lambda^2)\big)^{\frac1{2\alpha}+\frac12} \, dx \leq C\Big(1+\frac1\alpha\Big),
\end{equation}
where $C>0$ is a constant independent of $\alpha$ and $\lambda$.
\end{theorem}

\begin{proof}
We divide the proof into three steps.

\vskip10pt\noindent{\bf Step 1: Existence and uniqueness of a solution.}
According to \cite[Proposition 1]{Suquet}, there exists a unique pair $(\sigma^{\alpha,\lambda},v^{\alpha,\lambda})$ with the regularity \eqref{eq:sigmaal} and \eqref{eq:val} solving the system \eqref{NH-pb2}. Note that, in the statement of that result, the solution is only $H^1$ regular in time. However, a careful inspection of the proof shows that the time regularity is actually $W^{1,\infty}$ (see, e.g., \cite[page 402, line below (3.27)]{Suquet}).

From now on, in order to not overburden notation, we omit the super/subscripts $(\alpha,\lambda)$ and simply write $\sigma\equiv \sigma^{\alpha,\lambda}$, $v\equiv v^{\alpha,\lambda}$, $\gamma\equiv \gamma_{\alpha,\lambda}$. Unless otherwise stated, the letter $C$ denotes a positive constant independent of the parameters $(\alpha,\lambda)$. We also denote by $Q_t:=\Om \times (0,t)$ the space-time cylinder.

\vskip10pt\noindent{\bf Step 2: $L^2$-bounds for $(\sigma,v)$ and $(\dot\sigma,\dot v)$.} 
In order to prove \eqref{est1-sN2}, we proceed as follows. We (continuously) extend $(\sigma,v)$, $w$, $f$, $g$, and $\rho$ for negative times by setting
$$
(\sigma(t),v(t))=(\sigma_0,v_0), \quad w(t)=w(0)+t\dot w(0), \quad f(t)=f(0), \quad g(t)=g(0), \quad \rho(t)=\rho(0) 
$$
for all $t<0$. For a given function $\zeta$ defined on $(-\infty,T)$ with values in a vector space, we define the difference quotient
$$
\partial_t^h \zeta(\tau):=\frac{\zeta(\tau)-\zeta(\tau-h)}{h} \quad \text{ for }\tau<T \text{ and }h>0.
$$
By taking the increments of system \eqref{NH-pb2} we can write for a.e.\ $t\in [0,T]$
\begin{equation}
\label{eq:incremented}
\begin{cases}
\A\partial_t^h\dot \sigma(t) +\partial_t^h D\gamma(\sigma(t))=E \partial_t^hv(t) +\frac1h {\bf 1}_{(0,h)}(t) Ev_0  &  \text{ in } \Om,
\smallskip\\
\partial_t^h \dot v(t) - \Div \partial_t^h\sigma(t)=\partial_t^hf(t) &  \text{ in } \Om,
\smallskip\\
\partial_t^h v(t)=\partial_t^h \dot w(t) & \text{ on }\Gamma_D,
\smallskip\\
\partial_t^h\sigma(t)\nu=\partial_t^h g(t) & \text{ on }\Gamma_N.
\end{cases}
\end{equation}
If we test the second equation with $\partial_t^hv(t)-\partial_t^h\dot w(t) \in H^1_{\Gamma_D}(\Om;\R^n)$, integrate by parts and use \eqref{eq:LL}, we get
\begin{multline*}
\int_\Om\partial_t^h\dot v(t)\cdot \left(\partial_t^h v(t)-\partial_t^h\dot w(t)\right)\, dx+\int_\Om \partial_t^h \sigma(t)\cdot \left(E\partial_t^h v(t)-E\partial_t^h\dot w(t)\right) dx\\
=\pscal{\partial_t^h \LL(t)}{\partial_t^h v(t)-\partial_t^h\dot w(t)}=\int_\Om \partial_t^h \rho(t)\cdot \left(E\partial_t^h v(t)-E\partial_t^h\dot w(t)\right) dx.
\end{multline*}
Integrating in time on $[0,t]$ and using that $\partial_t^h v(0)=0$, we deduce 
\begin{eqnarray}
\frac{1}{2}\int_\Om |\partial_t^h v(t)|^2\,dx &=&\iint_{Q_t}\left(\partial_t^h \dot v\cdot \partial_t^h \dot w-
\left(\partial_t^h\sigma-\partial_t^h \rho\right)\cdot \left(E\partial_t^h v-E\partial_t^h\dot w\right)\right)dx \,ds \nonumber \\
&=&\int_\Om \partial_t^h v(t)\cdot \partial_t^h \dot w(t)\, dx \nonumber  \\
&&-\iint_{Q_t} \left(\partial_t^h v\cdot \partial_t^h \ddot w+\left(\partial_t^h\sigma-\partial_t^h \rho\right)\cdot\left(E\partial_t^h v-E\partial_t^h \dot w\right)\right) dx\,ds.
\label{first-v}
\end{eqnarray}
Testing the first equation in \eqref{eq:incremented} with $\partial_t^h \sigma(t)$ we get 
\begin{multline*}
\int_\Om\A \partial_t^h \dot \sigma(t)\cdot \partial_t^h\sigma(t)\, dx+\int_\Om\partial_t^h D\gamma(\sigma(t))\cdot \partial_t^h\sigma(t)\, dx\\
=\int_\Om \partial_t^h\sigma(t)\cdot E\partial_t^h v(t)\, dx+\int_\Om \frac1h {\bf 1}_{(0,h)}(t) Ev_0 \cdot \partial_t^h \sigma(t)\, dx,
\end{multline*}
from which we deduce, since $\gamma$ is a convex function,
$$
\int_\Om\A \partial_t^h \dot \sigma(t)\cdot \partial_t^h\sigma(t)\, dx \le \int_\Om\partial_t^h\sigma(t)\cdot E\partial_t^h v(t)\, dx
+\int_\Om \frac1h {\bf 1}_{(0,h)}(t) Ev_0 \cdot \partial_t^h \sigma(t)\, dx.
$$
Integrating in time on $[0,t]$ and using that $\partial_t^h\sigma(0)=0$, we get
\begin{equation}\label{first-sigma}
\frac{1}{2}\int_\Om \A \partial_t^h\sigma(t)\cdot \partial_t^h\sigma(t)\, dx \leq \iint_{Q_t}\partial_t^h\sigma\cdot E\partial_t^h v\, dx\,ds+ \frac1h \iint_{Q_t} {\bf 1}_{(0,h)}(s) Ev_0 \cdot \partial_t^h \sigma \, dx\,ds.
\end{equation}
By adding up \eqref{first-v} and \eqref{first-sigma} we get
\begin{multline*}
\frac{1}{2}\int_\Om |\partial_t^h v(t)|^2\,dx+\frac{1}{2}\int_\Om \A \partial_t^h\sigma(t)\cdot \partial_t^h\sigma(t)\, dx\\
   \le \int_\Om\partial_t^h v(t)\cdot \partial_t^h \dot w(t)\, dx+\iint_{Q_t} \partial_t^h\sigma\cdot E\partial_t^h \dot w\, dx\,ds-\iint_{Q_t} \partial_t^h v\cdot \partial_t^h \ddot w\, dx \,ds
   \\
+\iint_{Q_t} \partial_t^h \rho\cdot \left(E\partial_t^h v-E\partial_t^h \dot w\right)dx\,ds
+ \frac1h \iint_{Q_t} {\bf 1}_{(0,h)}(s) Ev_0 \cdot \partial_t^h \sigma \, dx\,ds.
\end{multline*}
Concerning the last term, we have 
\begin{multline*}
\left| 
 \frac1h \iint_{Q_t} {\bf 1}_{(0,h)}(s) Ev_0 \cdot \partial_t^h \sigma \, dx\,ds
\right|\le \frac{1}{h}\int_0^h \|Ev_0\|_{L^2(\Om)} \|\partial^h_t \sigma(s)\|_{L^2(\Om)}\,ds\\
\le \|Ev_0\|_{L^2(\Om)} \sup_{r\in[0,h]} \|\partial^h_t \sigma(r)\|_{L^2(\Om)}.
\end{multline*}
We set 
$$F^h(t):=\int_\Om |\partial_t^h v(t)|^2\,dx+\int_\Om \A \partial_t^h\sigma(t)\cdot \partial_t^h\sigma(t)\, dx,$$
so that by the Young inequality, for all $\e>0$ and for all $t \in [0,T]$
\begin{equation}\label{eq:1116}
F^h(t) \leq C_\e+C \int_0^t F^h(s)\, ds +\iint_{Q_t} \partial_t^h \rho\cdot E\partial_t^h v\, dx\, ds +\e \sup_{r \in [0,h]} F^h(r)
\end{equation}
for some constants $C>0$ and $C_\e>0$ depending on $\rho$, $w$, $\sigma_0$, and $v_0$, but independent of $(\alpha,\lambda,h)$. 

We now bound the space-time integral in the right-hand side of \eqref{eq:1116}. We note that
$$
\partial_t^h \rho\cdot E\partial_t^h v = \frac1n  {\rm tr}(\partial_t^h \rho)\Div\partial_t^h v
+ \partial_t^h \rho_D\cdot E\partial_t^h v.
$$
Since  $\partial_t^h D\gamma(\sigma(t))$ is a deviatoric matrix, the first equation of \eqref{eq:incremented} gives 
$$
\Div \partial_t^h v(t)={\rm tr}(\A \partial_t^h\dot\sigma(t))-\frac1h {\bf 1}_{(0,h)}(t) \Div v_0,
$$ 
so that, integrating by parts in time yields 
\begin{eqnarray*}
\lefteqn{\iint_{Q_t} {\rm tr}(\partial_t^h \rho(s))\Div\partial_t^h v(s)\, dx\,ds}
\\
& = & - \iint_{Q_t} {\rm tr}(\partial_t^h \dot \rho(s)){\rm tr}(\A \partial_t^h\sigma(s))\, dx\,ds
+\int_\Om {\rm tr}(\partial_t^h\rho(t)){\rm tr}(\A \partial_t^h\sigma(t))\,dx
\\
&&
 -\frac1h \iint_{Q_t} {\bf 1}_{(0,h)}(s) {\rm tr}(\partial_t^h \rho(s))\Div v_0 \, dx\,ds.
\end{eqnarray*}
By the Cauchy inequality and the assumption \eqref{eq:rhot} on $\rho$ we obtain
$$
\Big|\iint_{Q_t} {\rm tr}(\partial_t^h \rho(s))\Div\partial_t^h v(s)\, dx\,ds\Big|\leq C+ \int_0^t F^h(s)\,ds
+\frac12 F^h(t)
$$
for all $t\in[0,T]$. On the other hand using the time regularity \eqref{eq:rhotD} of $\rho_D$
\begin{eqnarray*}
\iint_{Q_t} \partial_t^h \rho_D(s)\cdot E\partial_t^h v(s)\, dx\,ds
& =& \int_0^{t-h} \int_\Om \frac{2\rho_D(s)-\rho_D(s-h)-\rho_D(s+h)}{h^2} \cdot Ev(s)\, dx\,ds\\
&& +\frac{1}{h}\int_{t-h}^t \int_\Om \frac{\rho_D(s)-\rho_D(s-h)}{h}\cdot Ev(s)\, dx\,ds \\
&& -\frac1h\int_0^h \int_\Om \frac{\rho_D(s)-\rho_D(0)}{h}\cdot Ev_0\, dx\,ds\\
& \le & C\sup_{s \in [0,t]} \int_\Om |E v(s)|\,dx,
\end{eqnarray*}
where $C>0$ depends on $\rho_D$ in view of \eqref{eq:rhotD}, but is independent of $(\alpha,\lambda,h)$. Applying the above estimates in \eqref{eq:1116} leads to
$$F^h(t) \leq C_\e+C\left(\int_0^t F^h(s)\, ds +\sup_{s \in [0,t]} \int_\Om |E v(s)|\,dx\right)+\e \sup_{r \in [0,h]} F^h(r) \quad \text{ for all }t \in [0,T].$$
By the Gronwall Lemma we infer that, for all $t \in [0,T]$,
$$F^h(t) \leq e^{C T} \bigg(C_\e+C \sup_{s \in [0,t]}\|Ev(s)\|_{L^1(\Om)} + \e \sup_{r \in [0,h]} F^h(r) \bigg).$$
Choosing $\e$ small enough leads to the following estimate
$$\sup_{s \in [0,t]}F^h(s) \leq C \bigg(1+ \sup_{s \in [0,t]}\|Ev(s)\|_{L^1(\Om)}  \bigg),$$
and letting $h\to 0$ 
\begin{equation}
\label{eq:a}
\esssup_{s \in [0,t]} \left(\|\dot\sigma(s)\|^2_{L^2(\Om)}+\|\dot v(s)\|^2_{L^2(\Om)}\right) \le C \bigg(1+\sup_{s \in [0,t]} \|E v(s)\|_{L^1(\Om)}\bigg).
\end{equation}

We next bound $Ev$ in $L^\infty(0,T;L^1(\Om;\Mn))$. Thanks to \eqref{Dg-2}, the fact that $D\gamma(\sigma)$ is a deviatoric matrix and the safe load condition \eqref{eq:rhoK}, we have 
for a.e.\ $s \in [0,t]$
\begin{multline*}
\int_\Om \left(Ev(s)-\A \dot \sigma(s)\right)\cdot \left(\sigma(s)-\rho(s)\right) dx=\int_\Om D\gamma(\sigma(s))\cdot \sigma(s)\, dx-\int_\Om D\gamma(\sigma(s))\cdot \rho(s)\, dx\\
\ge r_K \int_\Om |D\gamma(\sigma(s))|\,dx-\|\rho_D(s)\|_{L^\infty(\Om)} \int_\Om |D\gamma(\sigma(s))|\,dx
\ge c\int_\Om |D\gamma(\sigma(s))|\,dx,
\end{multline*}
where $c>0$. Since $|Ev(s)|\le |\A\dot\sigma(s)|+|D\gamma(\sigma(s))|$ we obtain
$$
\int_\Om |Ev(s)|\,dx\le \int_\Om |\A\dot\sigma(s)|\,dx+\frac{1}{c}\int_\Om \left(Ev(s)-\A \dot \sigma(s)\right)\cdot \left(\sigma(s)-\rho(s)\right) dx.
$$
In order to estimate $\pscal{\sigma(s)-\rho(s)}{Ev(s)}$, we test the second equation in \eqref{NH-pb2} with $v(s)-\dot w(s) \in H^1_{\Gamma_D}(\Om;\R^n)$ obtaining
$$
\int_\Om \dot v(s)\cdot \left(v(s)-\dot w(s)\right)\, dx+\int_\Om \sigma(s)\cdot \left(Ev(s)-E\dot w(s)\right) dx=\int_\Om \rho(s)\cdot \left(Ev(s)-E\dot w(s)\right) dx,
$$
so that
$$
\int_\Om \left(\sigma(s)-\rho(s)\right) \cdot Ev(s)\, dx=-\int_\Om \dot v(s)\left( v(s)-\dot w(s)\right) dx+\int_\Om \left(\sigma(s)-\rho(s)\right) \cdot E\dot w(s)\, dx.
$$
We conclude that for all $\e>0$
\begin{multline}
\label{eq:b}
\int_\Om |Ev(s)|\,dx\le \int_\Om |\A\dot\sigma(s)|\,dx+\frac{1}{c}\Big(-\int_\Om \A \dot \sigma(s)\cdot \left(\sigma(s)-\rho(s)\right) dx\\
-\int_\Om \dot v(s) \cdot\left(v(s)-\dot w(s)\right)dx +\int_\Om\left(\sigma(s)-\rho(s)\right)\cdot E\dot w(s)\, dx\Big)\\
\leq C_\e\left(1+\|\sigma(s)\|^2_{L^2(\Om)}+\|v(s)\|^2_{L^2(\Om)}\right)+\e\left(\|\dot \sigma(s)\|^2_{L^2(\Om)}+\|\dot v(s)\|^2_{L^2(\Om)}\right),
\end{multline}
where $C_\e>0$ is independent of $(\alpha,\lambda)$. Gathering \eqref{eq:a}, \eqref{eq:b}, and choosing $\e$ small enough, we infer
for all $t \in [0,T]$,
$$
\esssup_{s \in [0,t]}\left( \|\dot \sigma(s)\|^2_{L^2(\Om)}+\|\dot v(s)\|^2_{L^2(\Om)}\right) \le C\sup_{s \in [0,t]} \left(1+\|\sigma(s)\|^2_{L^2(\Om)}+\|v(s)\|^2_{L^2(\Om)}\right).
$$

On the other hand, for every $s\in [0,t]$
\begin{eqnarray}
\|\sigma(s)\|^2_{L^2(\Om)} & \le & \left\|\sigma_0+\int_0^s \dot \sigma(r)\,dr\right\|^2_{L^2(\Om)}\nonumber\\
& \le & 2\|\sigma_0\|^2_{L^2(\Om)}+2\left( \int_0^s \|\dot \sigma(r)\|_{L^2(\Om)}\,dr \right)^2\nonumber\\
& \le & C\left(1 + \int_0^t \|\dot \sigma(r)\|^2_{L^2(\Om)}\,dr\right)\label{eq:1505}
\end{eqnarray}
and a similar computation holds for $v$. We deduce that for a.e.\ $t \in [0,T]$
$$
 \|\dot \sigma(t)\|^2_{L^2(\Om)}+\|\dot v(t)\|^2_{L^2(\Om)}\le C\int_0^t\left(1+\| \dot \sigma(s)\|^2_{L^2(\Om)}+\|\dot v(s)\|^2_{L^2(\Om)}\right)\, ds.
$$
By the Gronwall Lemma and by \eqref{eq:1505} we finally obtain that for a.e.\ $t\in [0,T]$
\begin{equation}
\label{eq:est-L2sigma-v}
\|\dot \sigma(t)\|^2_{L^2(\Om)}+\|\dot v(t)\|^2_{L^2(\Om)}+ \| \sigma(t)\|^2_{L^2(\Om)}+\| v(t)\|^2_{L^2(\Om)}\le C,
\end{equation}
where $C>0$ is independent of $(\alpha,\lambda)$. This proves \eqref{est1-sN2}.

\vskip10pt\noindent{\bf Step 3: Further bounds.} In view of \eqref{eq:b} we deduce that
\begin{equation}
\label{eq:Ev-bound}
\sup_{s \in [0,T]} \int_\Om |Ev(s)|\,dx\le C,
\end{equation}
so that  \eqref{est1-vN2} follows by the Poincar\'e-Korn inequality and the boundary condition $v=\dot w$ on $\Gamma_D$.
\par
We now prove inequalities \eqref{est2-sN2}--\eqref{est2-uni2}. Multiplying the first equation of \eqref{NH-pb2} by $\sigma(t)$, the second equation of \eqref{NH-pb2} by $v(t)-\dot w(t) \in H^1_{\Gamma_D}(\Om;\R^n)$ and integrating by parts in space-time, we get that for all $0 \leq t_1 \leq t_2 \leq T$,
\begin{eqnarray*}
\lefteqn{\int_{t_1}^{t_2} \int_\Om D\gamma(\sigma)\cdot \sigma\, dx\,ds}
\\
 & = & \int_{t_1}^{t_2} \int_\Om\big(-\A \dot \sigma\cdot \sigma- \dot v\cdot (v-\dot w)+ (\sigma-\rho)\cdot E\dot w+\rho\cdot Ev\big)\, dx\, ds\\
 & = & \int_{t_1}^{t_2} \int_\Om\big(-\A \dot \sigma\cdot \sigma- \dot v\cdot (v-\dot w)+ (\sigma-\rho)\cdot E\dot w+\frac1n{\rm tr}(\rho)\Div v
 +\rho_D\cdot Ev \big)\, dx\, ds\\
& = & \int_{t_1}^{t_2} \int_\Om\big(-\A \dot \sigma\cdot \sigma- \dot v\cdot (v-\dot w)+ (\sigma-\rho)\cdot E\dot w+\frac1n{\rm tr}(\rho){\rm tr}(\A\dot\sigma)
 +\rho_D\cdot Ev \big)\, dx\, ds\\
& \leq & C(t_2-t_1),
 \end{eqnarray*}
where we used that $\Div v={\rm tr}(\A\dot\sigma)$ by the first equation in \eqref{NH-pb2}, the a priori estimates \eqref{eq:est-L2sigma-v} and \eqref{eq:Ev-bound}, as well as the assumptions \eqref{eq:rhot}--\eqref{eq:rhotD} on $\rho$ and \eqref{eq:w} on $w$.
We thus get
$$
\sup_{t \in [0,T]}\int_\Om D\gamma(\sigma(t))\cdot \sigma(t)\, dx \leq C.
$$
Taking into account inequalities \eqref{eq:Dg-1} and \eqref{Dg-2} for $D\gamma_{\alpha,\lambda}$ it follows that
$$\sup_{t \in [0,T]} \int_\Omega(1+ (d^2(\sigma(t)){\,\wedge\,} \lambda^2))^{\frac1{2\alpha}-\frac12}d^\ell(\sigma(t)) \, dx
\leq C \quad \text{ for } \ell=1,2,$$
while, by the convexity of $\gamma$ which gives $D\gamma(\xi)\cdot\xi \geq \gamma(\xi)-\gamma(0)$, 
$$\sup_{t\in[0,T]} \int_\Omega\gamma(\sigma(t)) \, dx \leq C+\frac{\alpha}{\alpha+1}\LL^n(\Om),$$
This inequality and the definition of $\gamma_{\alpha,\lambda}$ yield, in particular, that
$$
\sup_{t\in[0,T]} \int_\Omega \big(1+d^2(\sigma(t))\wedge \lambda^2)\big)^{\frac1{2\alpha}+\frac12} \, dx \leq C\Big(1+\frac1\alpha\Big).
$$
This concludes the proof.
\end{proof}

\begin{remark}\label{rem:energy-balance}
The solution provided by Theorem~\ref{thmNH-dyn} satisfies the following energy balance: for all $t \in [0,T]$
\begin{multline*}
\frac12 \int_\Om \mathbf A \sigma^{\alpha,\lambda}(t)\cdot \sigma^{\alpha,\lambda}(t)\, dx+\frac12 \int_{\Om} |v^{\alpha,\lambda}(t)|^2\, dx\\
 +\int_0^t\int_\Om \big( \gamma_{\alpha,\lambda}(\sigma^{\alpha,\lambda})+\gamma^*_{\alpha,\lambda}\left(D\gamma_{\alpha,\lambda}(\sigma^{\alpha,\lambda})\right) -\rho_D\cdot D\gamma_{\alpha,\lambda}(\sigma^{\alpha,\lambda}) \big)\, dx \, ds\\
= \frac12 \int_\Om \mathbf A \sigma_0\cdot \sigma_0\, dx+\frac12 \int_\Om|v_0|^2\, dx+\int_0^t\int_\Om \big( \dot v^{\alpha,\lambda}\cdot\dot w+ \rho\cdot ( \A\dot \sigma^{\alpha,\lambda}- E\dot w)+E\dot w\cdot \sigma^{\alpha,\lambda}\big)\, dx\, ds,
\end{multline*}
where $\gamma^*_{\alpha,\lambda}$ stands for the convex conjugate of $\gamma_{\alpha,\lambda}$.

Indeed, omitting again the dependence on $\alpha$ and $\lambda$, multiplying the first equation in \eqref{NH-pb2} by $\sigma$ and using the second equation we get for a.e.\ $s \in (0,T)$
\begin{eqnarray*}
\lefteqn{\int_\Om \A \dot \sigma(s)\cdot \sigma(s)\, dx+\int_\Om D\gamma(\sigma(s))\cdot \sigma(s)\, dx} \\
& = & \int_\Om\left(Ev(s)-E\dot w(s)\right)\cdot \sigma(s)\, dx+\int_\Om E\dot w(s)\cdot \sigma(s)\, dx\\
& = & -\int_\Om \dot v(s)\cdot \left(v(s)-\dot w(s)\right)dx+\pscal{\LL(s)}{v(s)-\dot w(s)}+\int_\Om E\dot w(s)\cdot \sigma(s)\, dx.
\end{eqnarray*}
The convexity and differentiability of $\gamma$ ensure that
$$D\gamma(\sigma(s))\cdot \sigma(s)=\gamma(\sigma(s))+\gamma^*(D\gamma(\sigma(s))) \quad \text{ a.e.\ in }\Om,$$
while \eqref{eq:LL} together with the fact that $D\gamma(\sigma) \in \MD$ a.e.\ in $\Om \times (0,T)$ imply that
\begin{eqnarray*}
\pscal{\LL(s)}{v(s)-\dot w(s)} & = & \int_\Om \rho(s)\cdot \left(Ev(s)-E\dot w(s)\right) dx\\
& = &\int_\Om \rho_D(s)\cdot D\gamma(\sigma(s))\, dx+\int_\Om \rho(s)\cdot \left(\A\dot \sigma(s)-E\dot w(s)\right) dx.
\end{eqnarray*}
In this way we arrive at the identity
\begin{multline*}
\int_\Om \dot v(s)\cdot v(s)\, dx +\int_\Om \A \dot \sigma(s)\cdot \sigma(s)\, dx+\int_\Om\big(\gamma(\sigma(s))+\gamma^*(D\gamma(\sigma(s)))-\rho_D(s)\cdot D\gamma(\sigma(s))\big)\,dx\\
= \int_\Om\big(\dot v(s)\cdot \dot w(s)+\rho(s)\cdot \left(\A\dot \sigma(s)- E\dot w(s)\right)+E\dot w(s)\cdot \sigma(s)\big)\, dx.
\end{multline*}
Integrating with respect to time between $0$ and $t$ yields the announced energy equality.
\end{remark}

We now prove higher spatial regularity of the solution $(\sigma^{\alpha,\lambda},v^{\alpha,\lambda})$ of the Norton-Hoff dynamic problem \eqref{NH-pb2}, with uniform estimates with respect to the parameters $\alpha$ and $\lambda$.

\begin{theorem}[\bf Higher spatial regularity]
\label{prop:reg-sN2}
In addition to the assumptions of Theorem~\ref{thmNH-dyn} suppose further that 
\begin{itemize}
\item $\partial K$ is of class $C^2$ and its second fundamental form is positive definite at every point of $\partial K$;
\item $\sigma_0\in H^1_{loc}(\Omega;\Mn)$;
\item $f\in L^2(0,T;H^1_{loc}(\Om;\R^n))$.
\end{itemize}
Then, for every open set $\omega\subset\subset\Omega$ there exists a constant $C_\omega>0$, depending on $\omega$ but independent of $(\alpha,\lambda)$, such that
\begin{equation}\label{prop-eq2}
\sup_{t\in[0,T]}\big( \|\nabla \sigma^{\alpha,\lambda}(t)\|_{L^2(\omega)}+ \|\nabla v^{\alpha,\lambda}(t)\|_{L^2(\omega)}\big) \leq C_\omega.
\end{equation}
\end{theorem}

\begin{proof}
Throughout the proof $\omega$ denotes an open set compactly supported in $\Omega$. Let $\omega'$ be a further open set such that
$ \omega \subset\subset \omega' \subset\subset \Omega$, and let $\varphi\in C^\infty_c(\omega';[0,1])$ be a cut-off function satisfying $\varphi=1$ in $\omega$. For every $1 \leq k \leq n$, $x \in \omega'$, and $h<{\rm dist}(\omega',\partial \Omega)$, we will use the notation
$$\partial_k^h \zeta(x)=\frac{\zeta(x+he_k)-\zeta(x)}{h}$$
for the difference quotient of a general function $\zeta$ defined on $\Omega$ with values in a vector space.

\par

As in the proof of Theorem~\ref{thmNH-dyn}, we omit the super/subscripts $(\alpha,\lambda)$ and simply write $\sigma\equiv \sigma^{\alpha,\lambda}$, $v\equiv v^{\alpha,\lambda}$, $\gamma\equiv \gamma_{\alpha,\lambda}$. We also denote by $Q_t:=\Om \times (0,t)$ the space-time cylinder.

\par

We divide the proof into two steps.

\vskip10pt\noindent{\bf Step 1: Bad non-uniform bounds.} 
We first show that
\begin{equation}\label{eq:1800}
\sigma \in L^\infty(0,T;H^1_{\rm loc}(\Omega;\Mn)).
\end{equation}
We have
$$
\A \partial_k^h \dot \sigma+ \partial_k^h D\gamma(\sigma)=\partial_k^h Ev
\quad \text{ a.e.\ in }\omega' \times (0,T).
$$
Multiplying the previous equation by $\varphi^4\partial_k^h \sigma$ and integrating over $\Omega$ yields for a.e.\ $t \in (0,T)$,
$$\int_\Om\varphi^4 \A \partial_k^h \dot \sigma(t)\cdot \partial_k^h \sigma(t) \, dx + \int_\Om \varphi^4 \partial_k^h D\gamma(\sigma(t))\cdot \partial_k^h \sigma(t) \, dx= \int_\Om \varphi^4 \partial_k^h Ev(t)\cdot\partial_k^h \sigma(t) \, dx.$$
We rewrite the right-hand side using the second equation of \eqref{NH-pb2} tested with $\varphi^4\partial_k^h v(t) \in H^1_0(\Om;\R^n)$ (note that it vanishes on the whole $\partial\Om$, so there are no boundary terms when integrating by parts): since
$$
\int_\Om \partial_k^h \dot v(t)\cdot \varphi^4\partial_k^h v(t)\, dx+\int_\Om \partial_k^h \sigma(t)\cdot E(\varphi^4\partial_k^h v(t))\, dx=\int_\Om \partial_k^h f(t)\cdot \varphi^4\partial_k^h v(t)\, dx
$$
we obtain
\begin{multline*}
\int_\Om \varphi^4 \A \partial_k^h \dot \sigma(t)\cdot \partial_k^h \sigma(t) \, dx+ \int_\Om \varphi^4 \partial_k^h D\gamma(\sigma(t))\cdot \partial_k^h \sigma(t)\, dx\\
 =-\int_\Om \partial_k^h \dot v(t)\cdot \varphi^4 \partial_k^h v(t)\, dx- \int_\Om \partial_k^h \sigma(t)\cdot \left(\partial_k^h v(t) \otimes \nabla \varphi^4\right) dx
 + \int_\Om \partial_k^h f(t)\cdot \varphi^4\partial_k^h v(t)\, dx.
 \end{multline*}
Integrating with respect to time yields for all $t \in (0,T)$,
\begin{multline}\label{eq:diff-quot}
\frac12\int_\Om \varphi^4 \A \partial_k^h \sigma(t)\cdot \partial_k^h \sigma(t) \, dx+\frac12\int_\Omega \varphi^4 \partial_k^h v(t)\cdot \partial_k^h v(t)\, dx\\
+\iint_{Q_t} \varphi^4 \partial_k^h D\gamma(\sigma)\cdot \partial_k^h \sigma \, dx\, ds
=  
\frac12\int_\Om \varphi^4 \A \partial_k^h \sigma_0\cdot \partial_k^h \sigma_0 \, dx+\frac12\int_\Omega \varphi^4 \partial_k^h v_0\cdot \partial_k^h v_0 \, dx \\
- \iint_{Q_t} \partial_k^h \sigma\cdot \left(\partial_k^h v\otimes \nabla \varphi^4 \right) dx\, ds
+\iint_{Q_t}\partial_k^h f\cdot \varphi^4\partial_k^h v\, dx\,ds.
\end{multline}
Owing to the convexity of $\gamma$, the third term in the left-hand side of the previous equation is nonnegative. Using that $(\sigma_0,v_0) \in H^1_{loc}(\Omega;\Mn) \times H^1(\Omega;\R^n)$, $v\in L^\infty(0,T;H^1(\Om;\R^n))$, and the H\"older inequality, we get that
\begin{multline*}
\int_\Om \varphi^4 \A \partial_k^h \sigma(t)\cdot \partial_k^h \sigma(t) \, dx+\int_\Omega \varphi^4 \partial_k^h v(t)\cdot \partial_k^h v(t)\, dx \\
 \le C\Bigg[1+\|\nabla \varphi\|_{L^\infty(\Omega)} \|\nabla v\|_{L^{\infty}(0,T;L^2(\Omega))} \left(\sup_{s \in [0,T]}\int_\Om \varphi^4 \A \partial_k^h \sigma(s)\cdot \partial_k^h \sigma(s) \, dx\right)^{1/2}\\
  + \|f\|^2_{L^2(0,T;H^1(\omega'))}+\|\nabla v\|^2_{L^{\infty}(0,T;L^2(\Omega))}  \Bigg],
\end{multline*}
where the constant $C>0$ is independent of $h$. Invoking the Cauchy inequality yields 
\begin{equation}
\label{eq:unif-alpha}
\sup_{t \in [0,T]} \int_\Om \varphi^4 \A \partial_k^h \sigma(t)\cdot \partial_k^h \sigma(t) \, dx \leq C_{\alpha},
\end{equation}
where $C_{\alpha}>0$ depends on $\alpha$ (through the $L^2$ norm of $\nabla v(t)$), but is independent of $h$. This proves \eqref{eq:1800}.

Since $D\gamma:\Mn \to \R$ is Lipschitz continuous, piecewise $C^1$, and $D\gamma(0)=0$, it follows that $D\gamma(\sigma)\in L^\infty(0,T;H^1_{loc}(\Om;\MD))$. As a consequence,
for a.e.\ $t\in [0,T]$
$$
\begin{cases}
\partial_k^h v \to \partial_k v & \text{ strongly in } L^2_{\rm loc}(\Omega;\R^n),\\
\partial_k^h \sigma \to \partial_k \sigma& \text{ strongly in } L^2_{\rm loc}(\Omega;\Mn),\\
\partial_k^hD\gamma(\sigma) \to \partial_k D\gamma(\sigma) & \text{ strongly in } L^2_{\rm loc}(\Omega;\MD),
\end{cases}$$
as $h \to 0$, and passing to the limit in \eqref{eq:diff-quot}, thanks to the uniform bound \eqref{eq:unif-alpha}, we infer
\begin{multline}\label{eq:estim1}
\frac12\int_\Om \varphi^4 \A \partial_k \sigma(t)\cdot \partial_k \sigma(t) \, dx+\frac12\int_\Omega \varphi^4 |\nabla v(t)|^2\, dx\\
+\iint_{Q_t} \varphi^4 \partial_k D\gamma(\sigma)\cdot \partial_k \sigma \, dx\, ds=
\frac12\int_\Om \varphi^4 \A \partial_k \sigma_0\cdot \partial_k \sigma_0 \, dx+\frac12\int_\Omega \varphi^4 |\nabla v_0|^2\, dx\\
 -\iint_{Q_t}  \partial_k \sigma\cdot \left(\partial_k v \otimes \nabla \varphi^4 \right) dx\, ds+ \iint_{Q_t} \partial_k f\cdot \varphi^4\partial_k v\, dx\,ds.
\end{multline}

\vskip10pt\noindent{\bf Step 2: Uniform bounds.} From now on, $C>0$ will always stand for a general constant independent of the parameters $(\alpha,\lambda)$. 
We start with an estimate for the term 
$$
\iint_{Q_t}  \partial_k \sigma\cdot \left(\partial_k v \otimes \nabla \varphi^4 \right) dx\, ds
$$
in the right-hand side of \eqref{eq:estim1}. By the second equation in \eqref{NH-pb2} we have that for all $1 \leq k \leq n$
$$
\dot v_k=(\Div \sigma)_k+ f_k=(\Div \sigma_D)_k + \frac1n \partial_k ({\rm tr}\sigma)+ f_k,
$$ 
hence
$$\partial_k \sigma_{ij} = \partial_k (\sigma_D)_{ij} + \frac1n \partial_k({\rm tr}\sigma) \delta_{ij}\\
= \Sigma_{ijk} + (\dot v_k- f_k ) \delta_{ij},$$
where
$$\Sigma_{ijk}:=\partial_k (\sigma_D)_{ij} - (\Div \sigma_D)_{k}\delta_{ij}.$$
Note for future use that $\Sigma_{ijk}$ is a linear function of the first partial derivative of $\sigma_D$ only. Using the summation convention over repeated indexes, we have
\begin{eqnarray*}
\lefteqn{\iint_{Q_t}  \partial_k \sigma\cdot \left(\partial_k v \otimes \nabla \varphi^4 \right) dx\, ds=  \iint_{Q_t} \partial_k \sigma_{ij}\partial_k v_i \partial_j \varphi^4\, dx \, ds}\\
&=&2\iint_{Q_t} \partial_k \sigma_{ij}(E v)_{ik} \partial_j \varphi^4\, dx \, ds - \iint_{Q_t} \partial_k \sigma_{ij}\partial_i v_k \partial_j \varphi^4\, dx \, ds\\
&=&2 \iint_{Q_t}\Sigma_{ijk}(E v)_{ik} \partial_j \varphi^4\, dx \, ds +2\iint_{Q_t} (E v)_{ik} (\dot v_k- f_k) \partial_i \varphi^4\, dx \, ds\\
&&- \iint_{Q_t} \partial_k \sigma_{ij}\partial_i v_k \partial_j \varphi^4\, dx \, ds.
\end{eqnarray*}
Integrating the last term by parts in space we get
\begin{eqnarray*}
\int_\Om\partial_k \sigma_{ij}\partial_i v_k \partial_j \varphi^4\, dx & =& -\langle \partial_{ki} \sigma_{ij},  v_k \partial_j \varphi^4\rangle_{H^{-1}(\omega')\times H^1_0(\omega')}
-\int_\Om v_k \partial_k\sigma_{ij}\partial_{ij}\varphi^4\,dx\\
& =& \int_\Om \partial_k v_k \partial_{i} \sigma_{ij}\partial_j\varphi^4\,dx +
\int_\Om v_k \partial_{i} \sigma_{ij}\partial_{jk}\varphi^4 dx
-\int_\Om v_k\partial_k\sigma_{ij}\partial_{ij}\varphi^4\,dx\\
& =& \int_\Om (Ev)_{kk} (\dot v_j- f_j)\partial_j\varphi^4\,dx
+\int_\Om v_k \partial_{i} \sigma_{ij}\partial_{jk}\varphi^4\,dx\\
&&-\int_\Om v_k\partial_k \sigma_{ij}\partial_{ij}\varphi^4\,dx,
\end{eqnarray*}
which yields
\begin{eqnarray}
\lefteqn{\iint_{Q_t}  \partial_k \sigma\cdot \left(\partial_k  v \otimes \nabla \varphi^4\right) dx\, ds} \nonumber
\\
& = & 2 \iint_{Q_t}\Sigma_{ijk}(E v)_{ik} \partial_j \varphi^4\, dx \, ds+2\iint_{Q_t} (E v)_{ik}(\dot v_k- f_k) \partial_i \varphi^4\, dx \, ds \nonumber \\
&& -\iint_{Q_t}  (Ev)_{kk} (\dot v_j- f_j)\partial_j\varphi^4\,dx\, ds
-\iint_{Q_t}  v_k \partial_{i} \sigma_{ij}\partial_{jk}\varphi^4\,dx\, ds \nonumber \\
&& +\iint_{Q_t}  v_k\partial_k \sigma_{ij}\partial_{ij}\varphi^4\,dx\, ds.
\label{hr-mg0}
\end{eqnarray}
The second term at the right-hand side can be bounded in the following way:
\begin{multline*}
\left|\iint_{Q_t}(E v)_{ik} (\dot v_k-f_k) \partial_i \varphi^4\,dx\, ds\right|\\
\le C\left[\iint_{Q_t} \varphi^4|\nabla v|^2\,dx\,ds+\|\dot v\|_{L^\infty(0,T; L^2(\Om))}^2+\|\varphi f\|_ {L^2(0,T;L^2(\Om))}^2\right].
\end{multline*}
The same estimate holds for the third term at the right-hand side of \eqref{hr-mg0}.
Since $\partial_{jk}\varphi^4=\varphi^2(12\partial_j\varphi \partial_k\varphi+4\varphi \partial_{jk}\varphi)$, we have
$$
\left|\iint_{Q_t}v_k \partial_{i} \sigma_{ij}\partial_{jk}\varphi^4\,dx\, ds\right|\le C\left[ \|v\|^2_{L^\infty(0,T;L^2(\Om))}+\iint_{Q_t} \varphi^4\A\partial_k \sigma\cdot \partial_k \sigma\, dx\,ds\right]
$$
and the same bound holds for the last term in \eqref{hr-mg0}. Applying the above estimates in \eqref{hr-mg0} and using that  $\A\dot\sigma+D\gamma(\sigma)=Ev$, we conclude that
\begin{eqnarray*}
\lefteqn{-\iint_{Q_t}  \partial_k \sigma\cdot \left(\partial_k v\otimes \nabla \varphi^4\right) dx\, ds}
\\
& \le & -2 \iint_{Q_t}\Sigma_{ijk} [\A \dot \sigma+ D\gamma(\sigma)]_{ik} \partial_j \varphi^4\, dx \, ds
+C\bigg(
 \|v\|^2_{W^{1,\infty}(0,T;L^2(\Omega))} + \| \varphi  f\|^2_{L^2(0,T;L^2(\Om))}
\\
&& {}+\iint_{Q_t}\varphi^4|\nabla v|^2\, dx \, ds  +\iint_{Q_t}\varphi^4\A\partial_k \sigma\cdot \partial_k \sigma\, dx\,ds
\bigg),
\end{eqnarray*}
so that from \eqref{eq:estim1} we obtain
\begin{multline}
\label{eq:estim2}
\frac12\int_\Om \varphi^4 \A \partial_k \sigma(t)\cdot \partial_k \sigma(t) \, dx+\frac12\int_\Omega \varphi^4 |\nabla v(t)|^2\, dx+\iint_{Q_t} \varphi^4 \partial_k D\gamma(\sigma)\cdot\partial_k \sigma \, dx\, ds\\
\leq C\bigg(\|\varphi \sigma_0 \|^2_{H^1(\Omega)}+\|v_0\|^2_{H^1(\Omega)} +  \|v\|^2_{W^{1,\infty}(0,T;L^2(\Omega))}
 +\|\varphi f\|_{L^2(0,T;L^2(\Om))}^2
 \\
{} +\|\varphi \nabla f\|_{L^2(0,T;L^2(\Om))}^2+\iint_{Q_t}\varphi^4|\nabla v|^2\, dx \, ds + \iint_{Q_t}\varphi^4 \A\partial_k \sigma\cdot \partial_k \sigma\, dx\, ds \bigg)\\
-2 \iint_{Q_t}\Sigma_{ijk} [\A \dot \sigma + D\gamma(\sigma)]_{ik} \partial_j \varphi^4\, dx \, ds.
\end{multline}

We now focus on the last term in the right-hand side of \eqref{eq:estim2}. Since $\Sigma_{ijk}$ 
 is a linear function of the first order partial derivatives of $\sigma_D$, thanks to the Cauchy inequality, we have
$$
 \iint_{Q_t}\Sigma_{ijk} \A \dot \sigma_{ik} \partial_j \varphi^4\, dx \, ds \leq C\left( \iint_{Q_t} \varphi^4 \A \partial_k \sigma\cdot \partial_k \sigma\, dx\, ds +\|\dot \sigma\|^2_{L^\infty(0,T;L^2(\Omega))}\right).
$$
Let us come to the term
$$
\iint_{Q_t}\Sigma_{ijk}  [D\gamma(\sigma)]_{ik} \partial_j \varphi^4\, dx \, ds.
$$
Since by \eqref{Dgamma}
$$
|D\gamma(\sigma)|=\left( 1+d^2(\sigma)\wedge \lambda^2\right)^{\frac{1}{2\alpha}-\frac{1}{2}}d(\sigma)
$$
while $|\Sigma_{ijk}| \leq C|\nabla \sigma_D|$, we may write
\begin{eqnarray}\label{eq:ref:intro}
&&\left| \iint_{Q_t} [D\gamma(\sigma)]_{ik}\Sigma_{ijk}\partial_j\varphi^4\, dx \, ds \right|\\
&&\hspace{2cm} \leq  C \iint_{Q_t} (1+d^2(\sigma) \wedge \lambda^2)^{\frac1{2\alpha}-\frac12}d(\sigma) |\nabla \sigma_D| |\partial_j \varphi^4|\, dx \, ds\nonumber
\\
&&\hspace{2cm}\leq   C \left(\int_0^t \int_{\{d(\sigma)\leq 1\}} \varphi^4 (1+d^2(\sigma) \wedge \lambda^2)^{\frac1{2\alpha}-\frac12}d(\sigma) |\nabla \sigma_D|^2\, dx \, ds\right)^{1/2}\nonumber
\\
&&\hspace{2.5cm} 
{}+ C \left(\int_0^t \int_{\{d(\sigma)> 1\}} \varphi^4 (1+d^2(\sigma) \wedge \lambda^2)^{\frac1{2\alpha}-\frac12} |\nabla \sigma_D|^2\, dx \, ds\right)^{1/2},\nonumber
\end{eqnarray}
where in the last step we used the H\"older inequality and \eqref{est2-sN2} with $\ell=1$ and $\ell=2$. Applying the Cauchy inequality, we get
\begin{eqnarray*}
&&\left| \iint_{Q_t} [D\gamma(\sigma)]_{ik}\Sigma_{ijk}\partial_j\varphi^4\, dx \, ds \right|\\
&&\hspace{2cm}\leq   \frac{C_K}{1+C_K} \int_0^t \int_{\{d(\sigma)\leq 1\}} \varphi^4 (1+d^2(\sigma) \wedge \lambda^2)^{\frac1{2\alpha}-\frac12}d(\sigma) |\nabla \sigma_D|^2\, dx \, ds
\\
&&\hspace{2.5cm} 
+  \frac{C_K}{1+C_K}\int_0^t \int_{\{d(\sigma)> 1\}} \varphi^4 (1+d^2(\sigma) \wedge \lambda^2)^{\frac1{2\alpha}-\frac12} |\nabla \sigma_D|^2\, dx \, ds +C,
\end{eqnarray*}
where $C_K>0$ is the same constant as in \eqref{eq:formD2gamma2}. Since $\partial K$ is of class $C^2$ and its second fondamental form is positive definite at every point, by \eqref{eq:formD2gamma2}
we have
$$ \partial_k D\gamma(\sigma)\cdot \partial_k \sigma  \geq (1+d^2(\sigma) \wedge \lambda^2)^{\frac{1}{2\alpha}-\frac12}\frac{C_K d(\sigma)}{1+C_K d(\sigma)}|\partial_k\sigma_D|^2,
$$
so that we conclude
\begin{multline*}
\frac12\int_\Om \varphi^4 \A \partial_k \sigma(t)\cdot\partial_k \sigma(t) \, dx+\frac12\int_\Omega \varphi^4 |\nabla v(t)|^2\, dx\\
+\iint_{Q_t} (1+d^2(\sigma) \wedge \lambda^2)^{\frac{1}{2\alpha}-\frac12}\frac{C_Kd(\sigma)}{1+C_Kd(\sigma)}\varphi^4 |\nabla\sigma_D|^2\, dx\,ds\\
\leq 
C\bigg(1+\|\varphi\sigma_0\|^2_{H^1(\Omega)}+\|v_0\|^2_{H^1(\Omega)} +  \|v\|^2_{W^{1,\infty}(0,T;L^2(\Omega))}+ \|\sigma\|^2_{W^{1,\infty}(0,T;L^2(\Omega))}
\\
+ \|\varphi f\|^2_{L^2(0,T;L^2(\Om))}+\|\varphi \nabla f\|^2_{L^2(0,T;L^2(\Om))}+\int_0^t \left(\int_\Om \varphi^4\A\partial_k \sigma\cdot \partial_k \sigma\, dx+ \int_{\Om}\varphi^4|\nabla v|^2\, dx\right) ds\bigg)\\
+\frac{C_K}{1+C_K} \int_0^t \int_{\{d(\sigma)\leq 1\}} (1+d^2(\sigma) \wedge \lambda^2)^{\frac1{2\alpha}-\frac12}d(\sigma) \varphi^4|\nabla \sigma_D|^2 \, dx \, ds\\
+\frac{C_K}{1+C_K} \int_0^t \int_{\{d(\sigma)> 1\}} (1+d^2(\sigma) \wedge \lambda^2)^{\frac1{2\alpha}-\frac12} \varphi^4 |\nabla \sigma_D|^2\, dx \, ds,
\end{multline*}
where the constant in front of the last two integrals was chosen in such a way that both terms are compensated by the third integral on the left-hand side. By the coercivity of $\A$ and \eqref{est1-sN2} we conclude that for all $t \in 0,T)$,
$$\|\varphi^2 \nabla \sigma(t)\|^2_{L^2(\Omega)}+\|\varphi^2 \nabla v(t)\|^2_{L^2(\Omega)} 
\leq C_1 \int_0^t \left(\|\varphi^2 \nabla \sigma\|^2_{L^2(\Omega)}+\|\varphi^2 \nabla v\|^2_{L^2(\Omega)}\right) ds + C_2,$$
where $C_1$, $C_2>0$ are independent of $(\alpha,\lambda)$. By the Gronwall Lemma \eqref{prop-eq2} follows, and the proof is concluded.
\end{proof}

\subsection{Existence of a strong solution in dynamical perfect plasticity}

In this section we recover the existence of strong evolutions in dynamic perfect plasticity by considering the Norton-Hoff approximations of the previous sections, and by letting the parameters $\alpha \to 0^+$ and $\lambda\to +\infty$.

\begin{theorem}[\bf Strong evolutions in dynamic perfect plasticity]
\label{thm:hreg2}
Under the assumptions of Theorem~\ref{prop:reg-sN2}, let $u_0\in H^1(\Om;\R^n)$ be such that
$\Div u_0={\rm tr}(\mathbf A\sigma_0)$ in $\Omega$. Then there exists a unique triplet $(u,e,p)$ with  regularity
$$
\begin{cases}
u \in W^{2,\infty}(0,T;L^2(\Omega;\R^n)) \cap W^{1,\infty}(0,T;H^1_{\rm loc}(\Omega;\R^n)) \cap W^{1,\infty}(0,T; BD(\Omega)),\\
e \in W^{1,\infty}(0,T;L^2(\Omega;\Mn)) \cap L^\infty(0,T;H^1_{\rm loc}(\Omega;\Mn)),\\
p \in W^{1,\infty}(0,T;\mathcal M(\Omega\cup \Gamma_D;\MD)) \cap W^{1,\infty}(0,T;L^2_{\rm loc}(\Omega;\MD))
\end{cases}
$$
satisfying the following conditions:

\medskip

\noindent {\bf Kinematic compatibility:} for all $t \in [0,T]$,
\begin{equation}\label{eq:kc}
Eu(t)=e(t)+p(t) \text{ $\LL^n$-a.e.\ in }\Om, \quad p(t)=(w(t)-u(t))\odot \nu \HH^{n-1} \text{ on }\Gamma_D;
\end{equation}

\medskip

\noindent {\bf Stress constraint:} for all $t\in [0,T]$,
\begin{equation}
\label{eq:sigmaK}
\sigma(t):=\mathbf C e(t) \in \K \quad \text{$\LL^n$-a.e.\ in }\Om;
\end{equation}

\medskip

\noindent {\bf Equation of motion:} for a.e.\ $t\in [0,T]$
\begin{equation}
\label{eq:motion}
\ddot u-\Div \sigma=f \quad \text{$\LL^n$-a.e.\ in }\Om;
\end{equation}

\medskip

\noindent {\bf Neumann boundary condition:} for a.e.\ $t\in (0,T)$,
\begin{equation}
\label{eq:Neumann}
\sigma(t)\nu=g(t)\quad  \HH^{n-1}\text{-a.e.\ on }\Gamma_N;
\end{equation}

\medskip

\noindent {\bf Flow rule:} for a.e.\ $t\in (0,T)$,
\begin{equation}\label{eq:fr}
\begin{cases}
H(\dot p(t))=\sigma_D(t)\cdot \dot p(t) & \text{$\LL^n$-a.e.\ in $\Om$,} \\
H((\dot w(t)-\dot u(t))\odot \nu)=(\sigma(t)\nu)_\tau \cdot (\dot w(t)-\dot u(t)) &\text{$\HH^{n-1}$-a.e.\ on $\Gamma_D$;}
\end{cases}
\end{equation}
\medskip

\noindent {\bf Initial condition:} 
\begin{equation}\label{eq:ic}
(u(0),\dot u(0),e(0),p(0))=(u_0,v_0,\mathbf A\sigma_0,Eu_0-\mathbf A\sigma_0).
\end{equation}
\end{theorem}

\begin{proof}
Let $\alpha_j\to 0^+$ and $\lambda_j\to +\infty$
and let us denote with $(\sigma_j,v_j)$ the evolution given by Theorem~\ref{thmNH-dyn} with the choice $\alpha:=\alpha_j$ and $\lambda=\lambda_j$, i.e., such that
\begin{equation}\label{NH-pb3}
\begin{cases}
\A\dot\sigma_j+D\gamma_{\alpha_j,\lambda_j}(\sigma_j)=Ev_j & \text{ in }\Omega \times (0,T),
\\
\dot v_j -\Div\sigma_j=f  & \text{ in }\Omega \times (0,T),
\\
v_j=\dot w & \text{ on }\Gamma_D \times (0,T),
\\
\sigma_j\nu=g & \text{ on }\Gamma_N \times (0,T),
\\
(\sigma_j(0),v_j(0))=(\sigma_0,v_0) & \text{ in }\Omega
\end{cases}
\end{equation}
We divide the proof into six steps.

\vskip10pt\noindent{\bf Step 1: Compactness.}
By applying the Ascoli-Arzel\`a Theorem we deduce from \eqref{est1-sN2} that 
there exists 
$$
(\sigma,v) \in W^{1,\infty}(0,T;L^2(\Omega;\Mn))\times W^{1,\infty}(0,T;L^2(\Omega;\R^n))
$$ 
such that, up to subsequences,
\begin{equation}\label{conv-sN2}
\begin{cases}
\sigma_j(t)\wto\sigma(t) & \text{weakly in } L^2(\Omega;\Mn),\\
v_j(t)\wto v(t) & \text{weakly in } L^2(\Omega;\R^n)
\end{cases}
\end{equation}
for every $t\in[0,T]$ and
\begin{equation}\label{wconv-sN2}
\begin{cases}
\sigma_j\wto\sigma & \text{weakly* in } W^{1,\infty}(0,T;L^2(\Omega;\Mn)),\\
v_j\wto v & \text{weakly* in } W^{1,\infty}(0,T;L^2(\Omega;\R^n)).
\end{cases}
\end{equation}
In addition, by \eqref{est1-vN2} we have that
\begin{equation}
\label{eq:bdv}
v_j\wto v \quad \text{weakly* in } L^\infty(0,T;BD(\Omega)).
\end{equation}
From Theorem~\ref{prop:reg-sN2} and the Urysohn property we also get that for all open sets $\omega\subset\subset\Omega$ and every $t\in[0,T]$
$$\begin{cases}
\sigma_j(t)\wto\sigma(t) \quad \text{weakly in }H^1(\omega;\Mn),\\
v_j(t)\wto v(t) \quad \text{weakly in } H^1(\omega;\R^n),
\end{cases}$$
so that $\sigma \in L^\infty(0,T;H^1_{\rm loc}(\Omega;\Mn))$ and $v \in L^\infty(0,T;H^1_{\rm loc}(\Omega;\R^n)) \cap L^\infty(0,T;BD(\Omega))$.

\vskip10pt\noindent{\bf Step 2: The limit evolution.}
We now set for every $t\in[0,T]$
$$
u(t):=u_0+\int_0^t v(s)\, ds
$$
as a	Bochner integral in $L^2(\Omega;\R^n)$. It follows that $u(0)=u_0$, $\dot u=v$ and
$$u \in W^{2,\infty}(0,T;L^2(\Omega;\R^n)) \cap W^{1,\infty}(0,T;H^1_{\rm loc}(\Omega;\R^n)) \cap W^{1,\infty}(0,T; BD(\Omega)).$$
We define 
$$e:=\A\sigma \in W^{1,\infty}(0,T;L^2(\Omega;\Mn)) \cap L^\infty(0,T;H^1_{\rm loc}(\Omega; \Mn)),$$
and, for all $t \in [0,T]$,
$$p(t):=
\begin{cases}
Eu(t)-e(t) & \text{ in }\Om,\\
(w(t)-u(t))\odot \nu \HH^{n-1} & \text{ on }\Gamma_D,
\end{cases}
$$
so that
$$p \in W^{1,\infty}(0,T;\mathcal M(\Omega\cup \Gamma_D;\Mn)) \cap W^{1,\infty}(0,T;L^2_{\rm loc}(\Omega;\Mn))$$
and the kinematic compatibility conditions \eqref{eq:kc} are satisfied. Moreover from the first equation in \eqref{NH-pb3} and taking into account \eqref{wconv-sN2} and \eqref{eq:bdv} we infer
\begin{equation}
\label{eq:conv-pdot}
D\gamma_{\alpha_j,\lambda_j}(\sigma_j)\wto \dot p\qquad\text{weakly* in }L^\infty(0,T;{\mathcal M}(\Om \cup \Gamma_D;\MD)).
\end{equation}
Together with the assumption $\Div u_0={\rm tr}(\mathbf A\sigma_0)$ in $\Omega$, this implies in particular that
$p(t)\in \MD$ for every $t\in[0,T]$.

\vskip10pt\noindent{\bf Step 3: The linear limit equations.}
By \eqref{conv-sN2} we have that $(\sigma(0),v(0))=(\sigma_0,v_0)$ and by the definition of $u$, $e$ and $p$, we also get that $(u(0),e(0),p(0))=(u_0,\mathbf A \sigma_0,Eu_0- \mathbf A \sigma_0)$, so that the initial condition \eqref{eq:ic} is satisfied. 

According to the second equation in \eqref{NH-pb3}, together with \eqref{wconv-sN2}, we have that 
 $\Div \sigma_j \wto \Div \sigma$ weakly* in $L^\infty(0,T;L^2(\Om;\R^n))$, so that $\sigma \in L^\infty(0,T;H_{\Div}(\Om;\Mn))$ and the equation of motion \eqref{eq:motion} is satisfied. Moreover, as $\sigma_j\nu \wto \sigma\nu$ weakly* in $L^\infty(0,T;H^{-\frac12}(\partial\Om;\R^n))$, we also get the validity of the Neumann boundary condition \eqref{eq:Neumann}.

\vskip10pt\noindent{\bf Step 4: The stress constraint.} 
Let  us fix $q>1$ and $\lambda>0$. By the H\"older inequality we have that for $j$ large and all $t \in [0,T]$,
\begin{eqnarray*}
\lefteqn{\int_\Omega \big[(1+d^2(\sigma_j(t))\wedge \lambda^2)^{q}+q(1+\lambda^2)^{q-1}(d^2(\sigma_j(t))-\lambda^2)_+\big] \, dx}\\
& \leq & \int_\Omega \big[(1+d^2(\sigma_j(t))\wedge \lambda_j^2)^{q}+q(1+\lambda_j^2)^{q-1}(d^2(\sigma_j(t))-\lambda_j^2)^+\big] \, dx\\
& \leq & \left(\int_\Omega (1+d^2(\sigma_j(t))\wedge \lambda_j^2)^{\frac{1}{2\alpha_j} + \frac12}\, dx\right)^{\frac{2q\alpha_j}{\alpha_j+1}}\LL^n(\Omega)^{1-\frac{2q\alpha_j}{\alpha_j+1}}\\
&& +\frac{q}{(1+\lambda_j^2)^{\frac{1}{2\alpha_j}+\frac12 -q}}\int_\Omega (1+\lambda_j^2)^{\frac{1}{2\alpha_j}-\frac12}(d^2(\sigma_j(t))-\lambda_j^2)^+\,dx.
\end{eqnarray*}
Using next the uniform estimates \eqref{eq:energy-bound} and \eqref{est2-uni2}, we infer that
\begin{multline*}
\int_\Omega \big[(1+d^2(\sigma_j(t))\wedge \lambda^2)^{q}+q(1+\lambda^2)^{q-1}(d^2(\sigma_j(t))-\lambda^2)_+\big] \, dx\\
\leq \left( C\frac{\alpha_j+1}{\alpha_j}\right)^{\frac{2q\alpha_j}{\alpha_j+1}} \LL^n(\Omega)^{1-\frac{2q\alpha_j}{\alpha_j+1}}+\frac{2qC}{(1+\lambda_j^2)^{\frac{1}{2\alpha_j}+\frac12 -q}}.
\end{multline*}
By convexity of the function $\xi \mapsto (1+d^2(\xi)\wedge \lambda^2)^{q}+q(1+\lambda^2)^{q-1}(d^2(\xi)-\lambda^2)_+$ 
 and the weak convergence  \eqref{conv-sN2}, we have for every $t\in [0,T]$, $q>1$ and $\lambda>0$
$$\int_\Omega \big[(1+d^2(\sigma(t))\wedge \lambda^2)^{q}+q(1+\lambda^2)^{q-1}(d^2(\sigma(t))-\lambda^2)_+\big] \, dx\leq \LL^n(\Omega).$$
Letting first $\lambda \to \infty$, the Fatou Lemma yields
$$\int_\Omega (1+d^2(\sigma(t))^{q} \, dx\leq \LL^n(\Omega).$$
Passing then to the limit as $q \to \infty$, we deduce that $d(\sigma(t))=0$ $\LL^n$-a.e. in $\Omega$, hence the stress constraint \eqref{eq:sigmaK}.

\vskip10pt\noindent{\bf Step 5: The flow rule.} 
In order to conclude, it remains to prove that the flow rule holds true. To this aim we will show that for a.e. $t\in [0,T]$
\begin{equation}
\label{eq:equalityH}
H(\dot p(t))=[\sigma_D(t)\cdot \dot p(t)]\qquad \text{as measures on $\Om\cup \Gamma_D$},
\end{equation}
where the left-hand side is interpreted as a convex function of measure, while the right-hand side stands for the duality pairing between $\sigma_D(t)$ and $\dot p(t)$ (see Definition~\ref{def:duality}). Note that for a.e.\ $t\in[0,T]$ we have $\sigma(t)\in \mathcal S^{\rm dyn}_g$  by \eqref{eq:sigmaK}--\eqref{eq:Neumann} and $(\dot u(t), \dot e(t), \dot p(t) )\in \mathcal A^{\rm dyn}_{\dot w(t)}$  by kinematic compatibility, so that the duality pairing is well defined.  

Equality \eqref{eq:equalityH} immediately implies the flow rule \eqref{eq:fr}. Indeed, since $\dot u(t) \in H^1_{loc}(\Om;\R^n)$ and $\dot p(t) \in L^2_{loc}(\Om;\MD)$, it follows from the integration by parts formula in Sobolev spaces that  $[\sigma_D(t)\cdot \dot p(t)]\res \Om=\sigma_D(t)\cdot \dot p(t)$ in $\Om$, so that by \eqref{eq:equalityH}
$$H(\dot p(t))=\sigma(t)\cdot \dot p(t) \quad \LL^n\text{-a.e. in }\Om.$$
On the other hand, according to \cite[Lemma 3.8]{FG}, we have that $[\sigma_D(t)\cdot \dot p(t)]\res \Gamma_D=(\sigma(t)\nu)_{\tau} \cdot (\dot w(t)-\dot u(t))$ on $\Gamma_D$, so that \eqref{eq:equalityH} on $\Gamma_D$ rewrite as
$$H((\dot w(t)-\dot u(t))\odot \nu)=(\sigma(t)\nu)_\tau \cdot (\dot w(t)-\dot u(t)) \quad \HH^{n-1}\text{-a.e. on }\Gamma_D.$$

Since $\sigma \in \mathbf K$ a.e. in $\Om \times (0,T)$, by \cite[Proposition 2.4]{DMDSM}, the inequality $H(\dot p(t)) \geq [\sigma_D(t)\cdot \dot p(t)]$ holds as measures in $\R^n$ for a.e. $t \in [0,T]$. Thus, in order to prove \eqref{eq:equalityH} it is sufficient to show that
\begin{equation}
\label{eq:ineq-check}
\hs(\dot p(t)):=\int_{\Om\cup \Gamma_D}H\left(\frac{\dot p(t)}{|\dot p(t)|}\right)\,d|\dot p(t)|=[\sigma_D(t)\cdot \dot p(t)](\R^n)=\pscal{\sigma_D(t)}{\dot p(t)}.
\end{equation}
We will obtain this relation by means of an energy inequality.
\par
First of all let us note that for every $\xi\in \MD$
$$
\gamma^*_{\alpha_j,\lambda_j}(\xi)=\sup\big\{ \tau\cdot \xi -\gamma_{\alpha_j,\lambda_j}(\tau)\big\}\ge \sup_{\tau \in \K}\big\{ \tau_D\cdot \xi -\gamma_{\alpha_j,\lambda_j}(\tau_D)\big\}=H(\xi)-\frac{\alpha_j}{\alpha_j+1}.
$$
Using the energy balance in Remark \ref{rem:energy-balance}, we get that 
\begin{multline}
\label{eq:Dgamma1}
\frac12 \int_\Om \mathbf A \sigma_j(t)\cdot\sigma_j(t)\, dx + \frac12 \int_{\Om} |v_j(t)|^2\, dx
+  \int_0^t \int_\Om \big( H\left(D\gamma_{\alpha_j,\lambda_j}(\sigma_j)\right)-  \rho_D\cdot D\gamma_{\alpha_j,\lambda_j}(\sigma_j) \big)\, dx \,  ds\\
\leq \frac12 \int_\Om \mathbf A \sigma_0\cdot \sigma_0\, dx+\frac12 \int_\Om|v_0|^2\, dx\\
+\int_0^t \int_\Om\big( \dot v_j\cdot \dot w+\rho\cdot( \A\dot \sigma_j- E\dot w)+E\dot w\cdot \sigma_j\big)\,dx\, ds+\frac{\alpha_j}{\alpha_j+1} \LL^n(\Om)T.
\end{multline}
We claim that
\begin{multline}
\label{eq:DMDEM-lsc}
 \int_0^t \big( \HH(\dot p(s))-\langle \rho_D(s),\dot p(s)\rangle \big)\, ds\\
 \le \liminf_{j\to+\infty}
\int_0^t \int_\Om \big( H(D\gamma_{\alpha_j,\lambda_j}(\sigma_j))-  \rho_D\cdot D\gamma_{\alpha_j,\lambda_j}(\sigma_j)\big)\, dx \,  ds.
\end{multline}
If the claim is true, then by \eqref{conv-sN2}, \eqref{wconv-sN2}, and \eqref{est1-sN2}
we can pass to the limit in \eqref{eq:Dgamma1} and obtain
\begin{multline*}
\frac12 \int_\Om \mathbf A \sigma(t)\cdot \sigma(t)\, dx + \frac12 \int_{\Om} |v(t)|^2\, dx + \int_0^t \big( \HH(\dot p)-\langle \rho_D,\dot p\rangle\big) \, ds\\
\leq \frac12 \int_\Om \mathbf A \sigma_0\cdot\sigma_0\, dx+\frac12 \int_\Om|v_0|^2\, dx+\int_0^t \int_\Om \big( \dot v\cdot \dot w+ \rho\cdot (\A\dot \sigma-E\dot w)+E\dot w\cdot \sigma\big)\,dx\, ds,
\end{multline*}
which can be rewritten as 
\begin{multline*}
\frac12 \int_\Om \mathbf A \sigma(t)\cdot\sigma(t)\, dx - \frac12 \int_\Om \mathbf A \sigma_0\cdot\sigma_0\, dx + \int_0^t \HH(\dot p)\, ds \\
\leq \int_0^t \langle \rho_D,\dot p\rangle\, ds +\int_0^t \int_\Om \big(-\dot v\cdot (v-\dot w)+\rho\cdot( \A\dot \sigma-E\dot w)+E\dot w\cdot \sigma\big)\, dx\, ds.
\end{multline*}
By assumption $\rho(t)\in \mathcal S^{\rm dyn}_g$ for every $t\in[0,T]$. Therefore, by applying \eqref{intbyp} first to $\rho$ and then to $\sigma$ 
(recall that from \eqref{eq:motion} we have $-\Div \sigma= f-\dot v$ in $\Om$) we have
\begin{multline*}
\pscal{\rho_D}{\dot p}
=\int_\Om\big( f\cdot (v-\dot w)+\rho\cdot (E\dot w- \A\dot \sigma )\big)\, dx +\int_{\Gamma_N}g\cdot (v-\dot w)\, d\HH^{n-1}\\
=\int_\Om\big( (f-\dot v)\cdot (v-\dot w)+\rho\cdot (E\dot w- \A\dot \sigma )\big)\, dx +\int_{\Gamma_N}g\cdot (v-\dot w)\, d\HH^{n-1} + \int_\Om \dot v \cdot (v-\dot w)\, dx\\
=\pscal{\sigma_D}{\dot p}+\int_\Om \big(\sigma\cdot (\A\dot \sigma- E\dot w)+\dot v\cdot (v-\dot w)+\rho\cdot (E\dot w- \A\dot\sigma)\big)\, dx
\end{multline*}
a.e.\ in $[0,T]$.
We thus get
$$
\frac{1}{2}\int_\Om \A\sigma(t)\cdot\sigma(t)\, dx -\frac{1}{2}\int_\Om\A\sigma_0\cdot\sigma_0\, dx+\int_0^t \hs(\dot p)\,ds
\le\int_0^t\big(\pscal{\sigma_D}{\dot p}+\pscal{\sigma}{\A\dot \sigma}\big)\,ds,
$$
so that
$$
\int_0^t \big(\hs(\dot p)-\pscal{\sigma_D}{\dot p}\big)\,ds\le 0.
$$
We conclude that $\hs(\dot p(t))=\pscal{\sigma(t)}{\dot p(t)}$ for a.e. $t\in (0,T)$, that is, \eqref{eq:ineq-check} holds true.
\par

We now prove the claim \eqref{eq:DMDEM-lsc}, adapting an argument of the proof of \cite[Theorem~3.3]{DMDSM} to our evolutionary setting. Let $\varphi\in C^\infty_c(\R^n)$ be such that $0\le \varphi \le 1$ and $\varphi=0$ in a neighborhood of $\Gamma_N$.
Notice that, being the integrand nonnegative,
\begin{multline}
\label{eq:ineq-Hphi}
\int_\Om\big( H(D\gamma_{\alpha_j,\lambda_j}(\sigma_j))-  \rho_D\cdot D\gamma_{\alpha_j,\lambda_j}(\sigma_j)\big)\,dx\\
\ge
\int_\Om\big( H(\varphi D\gamma_{\alpha_j,\lambda_j}(\sigma_j))-  \rho_D\cdot D\gamma_{\alpha_j,\lambda_j}(\sigma_j)\varphi\big)\, dx.
\end{multline}
By integration by parts  we may write
\begin{multline*}
\int_\Om \rho_D\cdot D\gamma_{\alpha_j,\lambda_j}(\sigma_j)\varphi \,dx=\int_\Om \rho \cdot( E\dot w-\A\dot\sigma_j)\varphi\,dx\\
+\int_\Om \rho\cdot [(\dot w-v_j)\odot \nabla \varphi)]\,dx -\int_\Om f\cdot (\dot w-v_j)\varphi\,dx,
\end{multline*}
so that by \eqref{est1-sN2}, \eqref{wconv-sN2}, and the Dominated Convergence Theorem
\begin{equation}
\begin{split}
& \lim_{j\to+\infty}\int_0^t \int_\Om \rho_D\cdot D\gamma_{\alpha_j,\lambda_j}(\sigma_j)\varphi \,dx
\\
& =  \int_0^t \left(\int_\Om \rho \cdot( E\dot w-\A\dot\sigma)\varphi\,dx+\int_\Om \rho\cdot [(\dot w-v)\odot \nabla \varphi)]\,dx -\int_\Om f\cdot (\dot w-v)\varphi\,dx\right)ds
\\
& =  \int_0^t \pscal{[\rho_D\cdot \dot p]}{\varphi}\,ds,\label{eq:dual-0t}
\end{split}
\end{equation}
where in the last equality we used \eqref{eq:duality-distr}.
On the other hand, since the convex functional $\mathcal H$ is sequentially weakly* lower semicontinuous on $\mathcal M(\Om\cup \Gamma_D; \MD)$, 
by \cite[Proposition 2.31]{AFP} there exists a sequence $\{\eta_k\}_{k\in\N}\subset C_0(\Om\cup \Gamma_D;\MD)$ such that 
$$
\mathcal H(q)= \sup_{k\in \N} \int_{\Om\cup \Gamma_D} \eta_k\cdot dq
$$
for every $q\in \mathcal M(\Om\cup \Gamma_D;\MD)$.
Now, using \eqref{eq:conv-pdot}, for every finite family $\{A_1,\dots, A_h\}$ of open disjoint intervals in $(0,t)$ we may write
\begin{eqnarray*}
\liminf_{j\to+\infty} \int_0^t \int_\Om H(\varphi D\gamma_{\alpha_j,\lambda_j}(\sigma_j(s)))\,dx\, ds
& \ge &
\liminf_{j\to+\infty} \sum_{k=1}^h \int_{A_k} \int_\Om H(\varphi D\gamma_{\alpha_j,\lambda_j}(\sigma_j(s)))\,dx\, ds
\\
& \ge & 
\liminf_{j\to+\infty} \sum_{k=1}^h \int_{A_k} \int_{\Om} \eta_k\cdot D\gamma_{\alpha_j,\lambda_j}(\sigma_j(s))\varphi \,dx\,ds
\\
& = & \sum_{k=1}^h \int_{A_k} \int_{\Om\cup \Gamma_D} \varphi\eta_k \cdot d\dot p(s)\,ds,
\end{eqnarray*}
so that, in view of Lemma~\ref{lemma:AFP} in the Appendix (applied to $U=(0,t)$, $\lambda=\mathcal L^1$ and $s \mapsto f_k(s):= \int_{\Om\cup \Gamma_D} \varphi\eta_k \cdot d\dot p(s)$) and of the positivity of $\mathcal H$, taking the sup over $k$ yields
\begin{equation}
\label{eq:Hphi}
\liminf_{j\to+\infty} \int_0^t \int_\Om H(\varphi D\gamma_{\alpha_j,\lambda_j}(\sigma_j(s))\,dx\, ds
\ge \int_0^t \hs(\varphi \dot p(s))\,ds.
\end{equation}
Gathering \eqref{eq:ineq-Hphi}--\eqref{eq:Hphi}, we conclude
\begin{multline*}
\liminf_{j\to+\infty} \int_0^t \int_\Om\big( H(D\gamma_{\alpha_j,\lambda_j}(\sigma_j))-  \rho_D\cdot D\gamma_{\alpha_j,\lambda_j}(\sigma_j)\big)\, dx\\
\ge 
\int_0^t \hs(\varphi \dot p(s))\,ds-\int_0^t \pscal{[\rho_D\cdot \dot p]}{\varphi}\,ds,
\end{multline*}
so that \eqref{eq:DMDEM-lsc} follows by letting $\varphi$ tend to ${\bf 1}_{\Om\cup \Gamma_D}$ (recall that $|[\rho_D\cdot \dot p]|\le C|\dot p|$ as measures).

\vskip10pt\noindent{\bf Step 6: Uniqueness.} The uniqueness of the solution is standard and follows, e.g., from \cite[Section~4.5]{BL} (see also \cite{BMo2}).
\end{proof}

\section{The quasi-static model}\label{sec:qst}

The regularity result for the dynamical model can be adapted to the quasi-static setting under some slightly different assumptions on the prescribed displacements and exterior loads, with some changes in the proofs. In this section we outline the main results. 

\medskip

\noindent $(H'_4)$  We  consider boundary displacements of the form
$$w\in H^1(0,T;H^1(\Omega;\R^n)).$$

\medskip

\noindent $(H'_5)$ As in $(H_5)$, we consider exterior loads $\LL(t)$ associated to a potential $\rho(t)$ according to \eqref{eq:LL}, with regularity
$$
\rho\in H^1(0,T;H_{\Div}(\Om;\Mn)), \quad
\rho_D\in W^{1,\infty}(0,T;L^\infty(\Om ;\MD)),
$$
and satisfying the uniform safe-load condition \eqref{eq:rhoK}.

\subsection{Norton-Hoff approximation} 

We start by proving an existence result for a Norton-Hoff approsimation of the quasistatic problem.

\begin{theorem}[\bf Quasi-static Norton-Hoff approximation]
\label{thm:psialpha}
Assume hypotheses $(H_1)$--$(H_3)$, $(H'_4)$, $(H'_5)$ and $(H_6)$. Then, for every $\alpha\in(0,1]$ and $\lambda>0$ the problem
\begin{equation}\label{NH-pb}
\begin{cases}
\A\dot\sigma^{\alpha,\lambda}+D\gamma_{\alpha,\lambda}(\sigma^{\alpha,\lambda})=Ev^{\alpha,\lambda} & \text{ in }\Omega \times (0,T),
\\
-\Div\sigma^{\alpha,\lambda}=f & \text{ in }\Omega\times (0,T),\\
v^{\alpha,\lambda}=\dot w & \text{ on } \Gamma_D \times (0,T), \\
\sigma^{\alpha,\lambda}\nu=g&\text{  on }\Gamma_N \times (0,T),\\
\sigma^{\alpha,\lambda}(0)=\sigma_0 &\text{ in }\Om
\end{cases}
\end{equation}
admits one and only one solution $(\sigma^{\alpha,\lambda}, v^{\alpha,\lambda})$ with
$$
\sigma^{\alpha,\lambda}\in H^1 (0,T;H_{\Div}(\Om;\Mn))
$$
and
$$
v^{\alpha,\lambda}\in L^2(0,T;H^1(\Om;\R^n)).
$$
Moreover, the following estimates hold:
\begin{equation}\label{est1-sN}
\sup_{t\in[0,T]} \|\sigma^{\alpha,\lambda}(t)\|_{L^2(\Om)}\leq C, \qquad  \int_0^T \|\dot\sigma^{\alpha,\lambda}(t)\|^2_{L^2(\Om)}\, dt  \leq C,
\end{equation}
\begin{equation}\label{est2-sN}
\int_0^T\int_\Omega \big(1+d^2(\sigma^{\alpha,\lambda}) \wedge \lambda^2\big)^{\frac1{2\alpha}-\frac12} d^\ell(\sigma^{\alpha,\lambda})\, dx\, dt\leq C \quad \text{ for } \ell=1,2,
\end{equation}
\begin{equation}\label{est2-uni}
\int_0^T \int_\Omega \big(1+d^2(\sigma^{\alpha,\lambda}) \wedge \lambda^2\big)^{\frac1{2\alpha}+\frac12} \, dx\, dt \leq C\Big(1+\frac1\alpha\Big),
\end{equation}
and
\begin{equation}\label{est1-vN}
\|v^{\alpha,\lambda}\|_{L^2(0,T;BD(\Omega))} \leq C,
\end{equation}
where $C$ is a constant independent of $\alpha$ and $\lambda$.
\end{theorem}

\begin{proof}
As before, in order to simplify notation, we omit  the explicit dependence on $\lambda$ and $\alpha$.
Lemma~\ref{lm:CL} in the Appendix guarantees the existence and uniqueness of a pair
$$
(\sigma,v)\in H^1(0,T; H_{\Div}(\Om;\Mn)) \times L^2(0,T; H^1(\Omega;\R^n))
$$
satisfying \eqref{NH-pb}. We now focus on the estimates. 
As before we will denote by $Q_t$ the space time cylinder $\Om\times (0,t)$.

\vskip10pt\noindent{\bf Step 1: Estimates for $\sigma$ and $\dot\sigma$.} Testing the first equation of \eqref{NH-pb} with $\dot \sigma$ and integrating in time on $[0,t]$ we get
\begin{multline}
\label{eq:energyS}
\iint_{Q_t}\A \dot \sigma \cdot \dot \sigma\, dx\,ds+\int_\Om \gamma(\sigma(t))\,dx-\int_\Om \gamma(\sigma_0)\,dx=\iint_{Q_t} Ev \cdot \dot \sigma\, dx \,ds\\
=\iint_{Q_t}\big( (Ev-E\dot w)\cdot \dot \sigma+E\dot w\cdot \dot\sigma\big)\, dx\,ds=
\iint_{Q_t} \big( (Ev-E\dot w)\cdot \dot \rho+E\dot w\cdot \dot\sigma\big)\, dx \,ds.
\end{multline}
Since  $D\gamma(\sigma)$ is a deviatoric matrix, the first equation of \eqref{NH-pb} gives $\Div v={\rm tr}(\A \dot\sigma)$, so that 
\begin{eqnarray*}
\int_\Om Ev\cdot \dot \rho\, dx & = & \int_\Om E_Dv\cdot \dot \rho_D\, dx+\frac{1}{n}\int_\Om (\Div v){\rm tr}(\dot \rho)\, dx \\
& = & \int_\Om E_Dv\cdot \dot \rho_D\, dx +\frac{1}{n}\int_\Om {\rm tr}(\A\dot\sigma) {\rm tr}(\dot \rho)\,dx,
\end{eqnarray*}
and we infer thanks to the assumptions $(H'_5)$ on $\rho$
$$
\iint_{Q_t} Ev\cdot \dot \rho\, dx \,ds\le C\int_0^t \left(\|Ev(s)\|_{L^1(\Om)}+\|\dot \rho(s)\|_{L^2(\Om)}\|\dot \sigma(s)\|_{L^2(\Om)}\right)\,ds.
$$
From \eqref{eq:energyS}, the regularity assumptions on $w$ and $\rho$ together with the Cauchy Inequality, we conclude that 
\begin{equation}
\label{eq:S1}
\int_0^t \|\dot \sigma(s)\|_{L^2(\Om)}^2\,ds\le C\left( 1+\int_0^t\|Ev(s)\|_{L^1(\Om)}\, ds\right).
\end{equation}
Let us estimate the right-hand side of \eqref{eq:S1}. Following computations similar to those of Step~2 in the proof of Theorem \ref{thmNH-dyn}, thanks to \eqref{Dg-2} and the safe load condition \eqref{eq:rhoK}, we have for a.e. $s \in (0,T)$
\begin{multline*}
\int_\Om (Ev(s)-\A \dot \sigma(s))\cdot (\sigma(s)-\rho(s))\, dx=\int_\Om D\gamma(\sigma(s))\cdot \sigma(s)\, dx -\int_\Om D\gamma(\sigma(s))\cdot \rho(s)\, dx\\
\ge r_K \int_\Om |D\gamma(\sigma(s))|\,dx-\|\rho_D(s)\|_{L^\infty(\Om)} \int_\Om |D\gamma(\sigma(s))|\,dx
\ge c\int_\Om |D\gamma(\sigma(s))|\,dx,
\end{multline*}
where $c>0$. Since $|Ev(s)|\le |\A\dot\sigma(s)|+|D\gamma(\sigma(s))|$, we obtain
\begin{equation}\label{eq:S1-1}
\int_\Om |Ev(s)|\,dx\le \int_\Om |\A\dot\sigma(s)|\,dx+\frac{1}{c}\int_\Om (Ev(s)-\A \dot \sigma(s))\cdot (\sigma(s)-\rho(s))\, dx.
\end{equation}
Let us estimate the term $\int_\Om Ev(s)\cdot (\sigma(s)-\rho(s))\, dx$. Since
$$
\int_\Om \sigma(s)\cdot (Ev(s)-E\dot w(s))\, dx =\int_\Om \rho(s)\cdot (Ev(s)-E\dot w(s))\, dx,
$$
we get
$$
\int_\Om (\sigma(s)-\rho(s))\cdot Ev(s)\, dx =\int_\Om (\sigma(s)-\rho(s))\cdot E\dot w(s)\, dx,
$$
so that, using again the assumptions on $\rho$, we obtain from \eqref{eq:S1-1} that
\begin{multline}
\label{eq:S3}
\int_\Om |Ev(s)|\,dx \le C \|\dot \sigma(s)\|_{L^2(\Om)} \\
+ C (\|\dot \sigma(s)\|_{L^2(\Om)} +\|E\dot w(s)\|_{L^2(\Om)}) (\|\sigma(s)\|_{L^2(\Om)}+\|\rho(s)\|_{L^2(\Om)}).
\end{multline}
Combining \eqref{eq:S1} and \eqref{eq:S3}, and using the Cauchy inequality we deduce
\begin{equation}\label{eq:S3-1}
\int_0^t \|\dot \sigma(s)\|_{L^2(\Om)}^2\,ds\le C\left( 1+\int_0^t \|\sigma(s)\|_{L^2(\Om)}^2\,ds\right).
\end{equation}
On the other hand, absolute continuity in time yields
$$
\|\sigma(t)\|_{L^2(\Om)}^2\le  C+2T\int_0^t \|\dot \sigma(s)\|_{L^2(\Om)}^2\,ds 
$$
for every $t\in[0,T]$, so that by \eqref{eq:S3-1}
$$
\|\sigma(t)\|_{L^2(\Om)}^2\le  C\left( 1+\int_0^t \|\sigma(s)\|_{L^2(\Om)}^2\,ds\right)
$$
for every $t\in[0,T]$.
By applying the Gronwall Lemma we finally deduce the first bound in \eqref{est1-sN}, which, in turn, together with \eqref{eq:S3-1}, implies the second bound in \eqref{est1-sN}.

\vskip10pt\noindent{\bf Step 2: Further estimates.} From \eqref{eq:S3} we deduce that 
$$
Ev \in L^2(0,T;L^1(\Om;\Mn)),
$$
so that \eqref{est1-vN} follows by the Poincar\'e-Korn inequality and the boundary condition $v=\dot w$ on~$\Gamma_D$. 
\par
We now prove \eqref{est2-sN} and \eqref{est2-uni}. We multiply the first equation of \eqref{NH-pb} by $\sigma$ and the second equation by $v$ and integrate in space. Using again that $\Div v={\rm tr}(\A \dot\sigma)$, we thus get that for a.e. $t \in (0,T)$
\begin{eqnarray*}
\int_\Om D\gamma(\sigma(t))\cdot \sigma(t)\, dx & = & -\int_\Om \A \dot \sigma(t) \cdot \sigma(t)\, dx+ \int_\Om \sigma(t)\cdot E\dot w(t)\, dx\\
&&+\int_\Om \rho(t)\cdot (Ev(t)-E\dot w(t))\, dx\\
& = & -\int_\Om \A \dot \sigma(t)\cdot \sigma(t)\, dx+ \int_\Om \sigma(t) \cdot E\dot w(t)\, dx +\int_\Om \rho_D(t)\cdot E_Dv(t)\, dx\\
&& {}+\frac{1}{n}\int_\Om {\rm tr}(\rho(t)){\rm tr}(\A\dot \sigma(t))\, dx-\int_\Om \rho(t)\cdot E\dot w(t)\, dx.
\end{eqnarray*}
It thus follows that
$$
\int_{0}^{T} \int_\Om D\gamma(\sigma)\cdot \sigma\, dx\,dt\le C.
$$
This immediately provides \eqref{est2-sN} by \eqref{eq:Dg-1} and \eqref{Dg-2}. On the other hand, the convexity of $\gamma$ ensures that $D\gamma(\xi)\cdot \xi \geq \gamma(\xi)-\frac{\alpha}{\alpha+1}$ for every $\xi\in\Mn$, hence
$$\int_0^T \int_\Omega\gamma(\sigma) \, dx\,dt \leq C+\frac{\alpha}{\alpha+1}T\LL^n(\Om),$$
from which \eqref{est2-uni} follows.
\end{proof}

\begin{remark}
As in the dynamical case (see Remark \eqref{rem:energy-balance}), the 
 following energy balance holds: for all $t \in [0,T]$,
\begin{multline*}
\frac12 \int_\Om \mathbf A \sigma^{\alpha,\lambda}(t)\cdot \sigma^{\alpha,\lambda}(t)\, dx\\
 +  \int_0^t\int_{\Om} \big(\gamma_{\alpha,\lambda}(\sigma^{\alpha,\lambda})+\gamma^*_{\alpha,\lambda}\left(D\gamma_{\alpha,\lambda}(\sigma^{\alpha,\lambda})\right) -\rho_D\cdot D\gamma_{\alpha,\lambda}(\sigma^{\alpha,\lambda})\big)\,dx \, ds\\
= \frac12 \int_\Om \mathbf A \sigma_0\cdot \sigma_0\, dx+\int_0^t\int_{\Om}\big(\rho\cdot (\A\dot \sigma^{\alpha,\lambda}- E\dot w)+E\dot w\cdot \sigma^{\alpha,\lambda}\big)\,dx\, ds,
\end{multline*}
where $\gamma^*_{\alpha,\lambda}$ stands for the convex conjugate of $\gamma_{\alpha,\lambda}$.
\end{remark}

We now prove higher regularity of the stress $\sigma^{\alpha,\lambda}$ with uniform estimates with respect to $(\alpha,\lambda)$ in dimension $n\leq 4$.

\begin{theorem}[\bf Higher spatial regularity for the stress in dimension $n\le 4$]
\label{prop:reg-sN}
In addition to the assumptions of Theorem~\ref{thm:psialpha}, suppose further that 
\begin{itemize}
\item $\partial K$ is of class $C^2$ and its second fundamental form is positive definite at every point of $\partial K$;
\item $\sigma_0\in H^1_{loc}(\Omega;\Mn)$;
\item $f\in L^\infty(0,T;W^{1,\infty}_{loc}(\Om;\R^n))$. 
\end{itemize}
If $n\le 4$, then for every open set $\omega\subset\subset\Omega$ there exists a constant $C_\omega>0$, depending on $\omega$ but independent of $\alpha$ and $\lambda$,
such that 
\begin{equation}\label{prop-eq}
\sup_{t\in[0,T]}\|\nabla \sigma^{\alpha,\lambda}(t)\|_{L^2(\omega)}\leq C_\omega.
\end{equation}
\end{theorem}

\begin{proof}
We divide the proof into three steps and, again, omit the explicit dependence on $(\lambda,\alpha)$. As before we will denote by $Q_t$ the space time cylinder $\Om\times (0,t)$.
\par
\vskip10pt\noindent{\bf Step 1: First regularity estimates.} 
Let us prove that 
\begin{equation}
\label{eq:sigmaH1loc}
\sigma \in L^2(0,T;H^1_{\rm loc}(\Omega;\Mn)).
\end{equation}
Throughout the proof $\omega$ denotes an open set compactly contained in $\Omega$. Let $\omega'$ be a further open set such that
$ \omega \subset\subset \omega' \subset \subset \Omega$, and let $\varphi\in  C^\infty_c(\omega';[0,1])$ be a cut-off function such that $\varphi=1$ in $\omega$. For every $1 \leq k \leq n$, $x \in \omega'$, and $h<{\rm dist}(\omega',\partial \Omega)$, we recall the notation
$$\partial_k^h \zeta(x)=\frac{\zeta(x+he_k)-\zeta(x)}{h}$$
for the difference quotient of a general function $\zeta$ defined on $\Omega$ with values in a vector space.
\par
We have for a.e. $t \in (0,T)$
$$
\A \partial_k^h \dot \sigma(t) + \partial_k^h D\gamma(\sigma(t))=\partial_k^h Ev(t) \quad \text{ a.e. in }\omega'.
$$
Multiplying the previous equation by $\varphi^6 \partial_k^h \sigma(t)$ and integrating over $\Omega$ yield
$$\int_\Om \varphi^6 \A \partial_k^h \dot \sigma(t)\cdot \partial_k^h \sigma(t)\, dx  + \int_\Om \varphi^6 \partial_k^h D\gamma(\sigma(t))\cdot\partial_k^h \sigma(t) \, dx= \int_\Om \varphi^6 \partial_k^h Ev(t)\cdot \partial_k^h \sigma(t) \, dx.$$

Recalling the summation convention over repeated indexes, we may write using the second equation in \eqref{NH-pb}
\begin{eqnarray*}
\int_\Om\varphi^6 \partial_k^h Ev(t)\cdot \partial _k^h \sigma(t)\, dx & = & -\int_\Om \partial_k^h v_i(t) \partial_j(\varphi^6 \partial_k^h \sigma_{ij}(t))\,dx\\
& = & -\int_\Om \partial_k^h v_i(t) (-\varphi^6 \partial_k^h f_i(t)+\partial_k^h \sigma_{ij}(t)\partial_j \varphi^6)\,dx,
\end{eqnarray*}
so that integrating in time
\begin{multline}
\label{eq:ineq-en0}
\frac{1}{2}\int_\Om \varphi^6 \A\partial_k^h \sigma(t)\cdot \partial_k^h \sigma(t)\, dx+
\iint_{Q_t}\varphi^6 \partial_k^h D\gamma(\sigma)\cdot \partial_k^h \sigma\, dx \,ds\\
=\frac{1}{2}\int_\Om\varphi^6 \A\partial_k^h \sigma_0\cdot \partial_k^h \sigma_0\, dx-\iint_{Q_t} \partial_k^h v_i (-\varphi^6 \partial_k^h f_i+\partial_k^h \sigma_{ij}\partial_j \varphi^6)\,dx\,ds.
\end{multline}
By the convexity of $\gamma_{\alpha,\lambda}$ we obtain
\begin{multline*}
\frac{1}{2}\int_\Om \varphi^6 \A \partial_k^h \sigma(t)\cdot \partial_k^h \sigma(t) \, dx
\le \frac{1}{2}\int_\Om \varphi^6 \A \partial_k^h  \sigma_0\cdot \partial_k^h \sigma_0 \, dx\\
-\iint_{Q_t} \partial_k^h v_i (-\varphi^6 \partial_k^h f_i+\partial_k^h \sigma_{ij}\partial_j \varphi^6)\,dx\, ds,
\end{multline*}
which yields
$$
\|\varphi^3 \partial_k^h\sigma(t)\|^2_{L^2(\Om)}
\le C\left(1+ \int_0^t \left( \|\partial_k^h v(s)\|_{L^2(\Om)}(\|f(s)\|_{H^1(\omega')}+\|\varphi^3 \partial_k^h \sigma(s)\|_{L^2(\Om)} \right)\right)ds,
$$
so that \eqref{eq:sigmaH1loc} follows by the Gronwall Lemma. Note that the previous argument does not provide an estimate independent of $(\alpha,\lambda)$ because of the presence of the $L^2$ norm of $\nabla v(s)$.

\vskip10pt\noindent{\bf Step 2: Refined estimates.} 
We now refine the argument of Step~1 to deduce a uniform estimate with respect to the parameters $\alpha$ and $\lambda$.

Since $D\gamma(\sigma(t))\in H^1_{loc}(\Om;\Mn)$ by composition (because $D\gamma$ is Lipschitz continuous), passing to the limit as $h\to 0^+$ in \eqref{eq:ineq-en0} we get
\begin{multline}
\label{eq:ineq-en1}
\frac{1}{2}\int_\Om \varphi^6\A\partial_k \sigma(t)\cdot \partial_k \sigma(t)\, dx+
\iint_{Q_t} \varphi^6\partial_k D\gamma(\sigma)\cdot \partial_k \sigma\, dx \,ds\\
=\frac{1}{2}\int_\Om \varphi^6\A\partial_k \sigma_0\cdot \partial_k \sigma_0\, dx -\iint_{Q_t} \partial_k v_i ( -\varphi^6\partial_k f_i +\partial_j \varphi^6\partial_k\sigma_{ij} )\,dx\,ds.
\end{multline}
The last term on the right-hand side of \eqref{eq:ineq-en1} can be rewritten as follows:
\begin{eqnarray*}
\lefteqn{-\int_\Om \partial_k v_i ( -\varphi^6\partial_kf_i +\partial_j \varphi^6\partial_k\sigma_{ij} )\,dx} \\
&=& -\int_\Om (2(Ev)_{ik}-\partial_i v_k)( -\varphi^6\partial_kf_i +\partial_j \varphi^6\partial_k\sigma_{ij} )\,dx\\
&=& -2\int_\Om  (\A\dot\sigma+D\gamma(\sigma))_{ik}( -\varphi^6\partial_kf_i +\partial_j \varphi^6\partial_k\sigma_{ij} )\,dx\\
&&-\int_\Om v_k (- \varphi^6\partial_{ik}f_i-2\partial_kf_i \partial_i \varphi^6+\partial_k\sigma_{ij}\partial_{ij}\varphi^6)\,dx\\
&=& -2\int_\Om  (\A\dot\sigma+D\gamma(\sigma))_{ik}( -\varphi^6\partial_kf_i +\partial_j \varphi^6\partial_k\sigma_{ij} )\,dx\\
&&+\int_\Om (\A\dot\sigma)_{kk}(-\varphi^6\partial_i f_i-2f_i\partial_i \varphi^6+\sigma_{ij}\partial_{ij}\varphi^6 )\,dx\\
&&-\int_\Om v_k(\partial_i f_i \partial_k \varphi^6+2f_i \partial_{ik}\varphi^6-\sigma_{ij}\partial_{ijk}\varphi^6)\,dx,
\end{eqnarray*}
where we used again the first equation in \eqref{NH-pb} and the fact that $D\gamma(\sigma)$ is deviatoric.
Integrating in time, by \eqref{est1-sN} we have that
$$\Big| \iint_{Q_t} (\A\dot\sigma)_{kk}\sigma_{ij}\partial_{ij}\varphi^6\, dx \, ds\Big|
\leq C\|\dot\sigma\|_{L^2(0,T;L^2(\Om))}\|\sigma\|_{L^2(0,T;L^2(\Om))}\leq C$$
and
$$
\Big| \iint_{Q_t} (\A\dot\sigma)_{kk}(\varphi^6\partial_i f_i +2f_i\partial_i \varphi^6) \, ds\Big|
\leq C\|\dot\sigma\|_{L^2(0,T;L^2(\Om))}\|f\|_{L^2(0,T;H^1(\omega'))}\leq C.
$$
By \eqref{est1-vN} and the continuous embedding of $BD(\Omega)$ into $L^{\frac{n}{n-1}}(\Omega;\R^n)$
we obtain
\begin{multline*}
\Big| \iint_{Q_t} v_k(\partial_if_i \partial_k \varphi^6+2f_i\partial_{ik}\varphi^6) \, dx \, ds\Big|
\\
\leq C\|v\|_{L^2(0,T;L^{\frac{n}{n-1}}(\Om))}\left(\|\Div f\|_{L^2(0,T;L^n(\omega'))}+\|f\|_{L^2(0,T;L^n(\omega'))}\right) \leq C.
\end{multline*}
Moreover, by the second estimate in \eqref{est1-sN} and the Cauchy inequality we deduce
$$\Big| \iint_{Q_t} (\A\dot\sigma)_{ik}\partial_k\sigma_{ij}\partial_j\varphi^6\, dx  \, ds\Big| \leq C + \iint_{Q_t} \varphi^6|\partial_k\sigma|^2\, dx\, ds,$$
where we also used that $\partial_j\varphi^6=6\varphi^5 \partial_j\varphi$,  while
$$
\Big| \iint_{Q_t} \varphi^6(\A\dot\sigma)_{ik}\partial_kf_i\, dx  \, ds\Big| \leq C \|\dot\sigma\|_{L^2(0,T;L^2(\Om))}\|\nabla f\|_{L^2(0,T;L^2(\omega'))}\leq C.
$$

Concerning the term
$$
\iint_{Q_t} (D\gamma(\sigma))_{ik}(-\varphi^6\partial_k f_i+\partial_k\sigma_{ij}\partial_j\varphi^6)\, dx \, ds,
$$
by \eqref{Dgamma} we have
$$
\Big|\iint_{Q_t} (D\gamma(\sigma)_{ik}\partial_k\sigma_{ij}\partial_j\varphi^6\, dx \, ds\Big|\le 
C \iint_{Q_t} \varphi^5 (1+d^2(\sigma)\wedge\lambda^2)^{\frac1{2\alpha}-\frac12}d(\sigma) |\partial_k\sigma_{ij}|\, dx \, ds.
$$
Since $\frac{1}{n}\nabla ({\rm tr}\,\sigma)=-f- \Div\sigma_D$ in $\Omega$, we deduce that $|\nabla \sigma|\le |f|+|\nabla \sigma_D|$, which yields
\begin{eqnarray*}
\lefteqn{\Big|\iint_{Q_t} (D\gamma(\sigma))_{ik}\partial_k\sigma_{ij}\partial_j\varphi^6\, dx \, ds\Big|}\\
& \le & 
C \iint_{Q_t} \varphi^5 (1+d^2(\sigma)\wedge \lambda^2)^{\frac1{2\alpha}-\frac12}d(\sigma) (|f|+|\nabla \sigma_D|)\, dx \, ds\\
& \le  & C\|f\|_{L^\infty(0,T;L^\infty(\omega'))} \iint_{Q_t} \varphi^5 (1+d^2(\sigma)\wedge\lambda^2)^{\frac1{2\alpha}-\frac12}d(\sigma) \,dx\, ds
\\
&&{} +C \iint_{Q_t} \varphi^5 (1+d^2(\sigma)\wedge \lambda^2)^{\frac1{2\alpha}-\frac12}d(\sigma) |\nabla \sigma_D|\, dx \, ds\\
&\le & C\left( 1+ \iint_{Q_t} \varphi^5 (1+d^2(\sigma)\wedge \lambda^2)^{\frac1{2\alpha}-\frac12}d(\sigma) |\nabla \sigma_D|\, dx \, ds\right)
\end{eqnarray*}
where in the last inequality we used \eqref{est2-sN} with $\ell=1$. Again by \eqref{Dgamma} and \eqref{est2-sN} with $\ell=1$ we have
\begin{eqnarray*}
\lefteqn{\Big| \iint_{Q_t}\varphi^6 (D\gamma(\sigma))_{ik}\partial_kf_i \,dx\, ds\Big|}
\\
& \le & \iint_{Q_t} \varphi^6 (1+d^2(\sigma)\wedge \lambda^2)^{\frac{1}{2\alpha}-\frac{1}{2}}d(\sigma)|\nabla f|\,dx\, ds\\
& \le & \|\nabla f\|_{L^\infty(0,T;L^\infty(\omega'))}\iint_{Q_t} \varphi^6 (1+d^2(\sigma)\wedge \lambda^2)^{\frac{1}{2\alpha}-\frac{1}{2}}d(\sigma)\,dx\, ds \leq C.
\end{eqnarray*}
The assumptions on $K$ and \eqref{eq:formD2gamma2} guarantee that
$$\int_\Om \varphi^6\partial_k(D\gamma(\sigma)\cdot\partial_k\sigma\,dx\ge 
C_K \int_\Om \varphi^6 (1+d^2(\sigma)\wedge \lambda^2)^{\frac{1}{2\alpha}-\frac{1}{2}} \frac{d(\sigma)}{1+C_Kd(\sigma)}|\nabla \sigma_D|^2\,dx,$$
so that by \eqref{eq:ineq-en1} and all the previous estimates we finally obtain that
\begin{multline}
\label{ine-gr}
C\|\varphi^3 \nabla \sigma(t)\|^2_{L^2(\Om)}
+C_K \iint_{Q_t} \varphi^6 (1+d^2(\sigma)\wedge \lambda^2)^{\frac{1}{2\alpha}-\frac{1}{2}} \frac{d(\sigma)}{1+C_Kd(\sigma)}|\nabla \sigma_D|^2\,dx\, ds\\
\le C+\int_0^t \|\varphi^3 \nabla \sigma(s)\|^2_{L^2(\Om)} \, ds+C \iint_{Q_t} \varphi^5 (1+d^2(\sigma)\wedge \lambda^2)^{\frac1{2\alpha}-\frac12}d(\sigma) |\nabla \sigma_D|\, dx \, ds\\
+\iint_{Q_t} v_k\sigma_{ij}\partial_{ijk}\varphi^6\,dx\, ds.
\end{multline}
Applying the Cauchy-Schwarz inequality, we may write
\begin{eqnarray*}
\lefteqn{\iint_{Q_t} \varphi^5 (1+d^2(\sigma)\wedge \lambda^2)^{\frac1{2\alpha}-\frac12}d(\sigma) |\nabla \sigma_D|\, dx \, ds}
\\
& = & \int_0^t\int_{\{d(\sigma)\le 1\}} \varphi^5 (1+d^2(\sigma)\wedge \lambda^2)^{\frac1{2\alpha}-\frac12}d(\sigma) |\nabla \sigma_D|\, dx \, ds\\
&&+\int_0^t\int_{\{d(\sigma)> 1\}} \varphi^5 (1+d^2(\sigma)\wedge \lambda^2)^{\frac1{2\alpha}-\frac12}d(\sigma) |\nabla \sigma_D|\, dx \, ds\\
& \le & \left(  \int_0^t\int_{\{d(\sigma)\le 1\}} \varphi^6 (1+d^2(\sigma)\wedge \lambda^2)^{\frac1{2\alpha}-\frac12}d(\sigma) |\nabla\sigma_D|^2\, dx \, ds\right)^{\frac{1}{2}} \times \\
&&\times \left(\int_0^t\int_{\{d(\sigma)\le 1\}} (1+d^2(\sigma)\wedge \lambda^2)^{\frac1{2\alpha}-\frac12}d(\sigma)\, dx \, ds\right)^{\frac{1}{2}}\\
&&{}+\left(  \int_0^t\int_{\{d(\sigma)>1\}} \varphi^6 (1+d^2(\sigma)\wedge \lambda^2)^{\frac1{2\alpha}-\frac12}|\nabla \sigma_D|^2\, dx \, ds\right)^{\frac{1}{2}}
\times \\
&&\times \left( \int_0^t\int_{\{d(\sigma)>1\}} (1+d^2(\sigma)\wedge \lambda^2)^{\frac1{2\alpha}-\frac12}d(\sigma)^2\, dx \, ds\right)^{\frac{1}{2}}.
\end{eqnarray*}
Thus, by the Cauchy inequality and by \eqref{est2-sN} for $\ell=1,2$ we can rewrite \eqref{ine-gr} in the form
\begin{multline*}
C\|\varphi^3 \nabla\sigma(t)\|^2_{L^2(\Om)}
+C_K \iint_{Q_t} \varphi^6 (1+d^2(\sigma)\wedge \lambda^2)^{\frac{1}{2\alpha}-\frac{1}{2}} \frac{d(\sigma)}{1+C_Kd(\sigma)}|\nabla \sigma_D|^2\,dx\, ds\\
\le C+\int_0^t \|\varphi^3 \nabla \sigma(s)\|^2_{L^2(\Om)}\, ds+
\frac{C_K}{1+C_K}  \int_0^t\int_{\{d(\sigma)\leq 1\}} \varphi^6(1+d^2(\sigma)\wedge \lambda^2)^{\frac1{2\alpha}-\frac12}d(\sigma) |\nabla \sigma_D|^2 \, dx \, ds\\
+\frac{C_K}{1+C_K} \int_0^t\int_{\{d(\sigma)> 1\}} \varphi^6 (1+d^2(\sigma)\wedge \lambda^2)^{\frac1{2\alpha}-\frac12} |\nabla \sigma_D|^2\, dx \, ds
+\iint_{Q_t} v_k\sigma_{ij}\partial_{ijk}\varphi^6\,dx\, ds,
\end{multline*}
so that we arrive at the inequality
\begin{equation}
\label{eq:ine-gr2}
\|\varphi^3 \nabla\sigma(t)\|^2_{L^2(\Om)} \le C\left( 1+\int_0^t \|\varphi^3 \nabla\sigma(s)\|^2_{L^2(\Om)}\, ds
+\iint_{Q_t} v_k\sigma_{ij}\partial_{ijk}\varphi^6\,dx\, ds
\right).
\end{equation}

\vskip10pt\noindent{\bf Step 3: Conclusion.} We claim that
\begin{equation}
\label{claim:234}
\left|\iint_{Q_t} v_k\sigma_{ij}\partial^3_{ijk}\varphi^6\,dx\, ds\right|\le 
C\left(1+\int_0^t \|\varphi^3\nabla \sigma(s)\|_{L^{2}(\Om)}^2\,ds\right),
\end{equation}
so that inequality \eqref{eq:ine-gr2} entails
$$
\|\varphi^3 \nabla\sigma(t)\|^2\le C\left( 1+\int_0^t \|\varphi^3 \nabla \sigma(s)\|^2_{L^2(\Om)}\, ds\right)
$$
and the result follows by the Gronwall Lemma.
\par
It is in the proof of \eqref{claim:234} that the restriction on the dimension $n\le 4$ comes in. Indeed, if $n \leq 4$, we have
$$
n\le 2^*=\frac{2n}{n-2}
$$
and by the H\"older Inequality and  Sobolev embedding
\begin{eqnarray*}
\left|\int_\Om v_k\sigma_{ij}\partial_{ijk}\varphi^6\,dx\right| & \le & C\int_\Om |v| |\varphi^3\sigma|\,dx\\
& \le & C\|v\|_{L^{\frac{n}{n-1}}(\Om)}\|\varphi^3\sigma\|_{L^n(\Om)}\\
& \le &C \|v\|_{L^{\frac{n}{n-1}}}(\Om)\|\varphi^3\sigma\|_{L^{2^*}(\Om)}\\
& \le & C \|v\|_{L^{\frac{n}{n-1}}(\Om)}\|\nabla (\varphi^3\sigma)\|_{L^{2}(\Om)}.
\end{eqnarray*}
Taking into account \eqref{est1-vN} we deduce
\begin{eqnarray*}
\int_0^t\left|\int_\Om v_k\sigma_{ij}\partial_{ijk}\varphi^6\,dx\right|\,ds & \le & C\int_0^t \left(\|v(s)\|_{L^{\frac{n}{n-1}}(\Om)}^2+\|\nabla(\varphi^3\sigma(s))\|_{L^{2}(\Om)}^2\right) ds\\
& \le & C\int_0^t \left(\|v(s)\|_{BD(\Om)}^2+\|\varphi^3\nabla\sigma(s)\|_{L^{2}(\Om)}^2+\|\sigma(s)\|_{L^2(\Om)}^2\right)ds\\
& \le & C\left(1+\int_0^t \|\varphi^3\nabla\sigma(s)\|_{L^{2}(\Om}^2\,ds\right).
\end{eqnarray*}
The proof is now complete.
\end{proof}

\subsection{Quasi-static evolutions in perfect plasticity}
We conclude this section by an existence result for quasi-static evolutions in perfect plasticity, which follows by passing to the limit in the quasi-static Norton-Hoff problem as $\alpha \to 0^+$ and $\lambda\to +\infty$.

\begin{theorem}[\bf Quasi-static evolutions in perfect plasticity]
\label{thm:hreg3}
Under the same assumptions of Theorem~\ref{prop:reg-sN}, given $(u_0,e_0,p_0)\in {\mathcal A}^{\rm qst}_{w(0)}$ with $e_0:=\A \sigma_0$, there exists a triplet
$$
(u,e,p)\in H^1(0,T; BD(\Omega){\times} L^2(\Omega;\Mn){\times} {\mathcal M}(\Omega \cup \Gamma_D;\MD))
$$
satisfying the following condition:

\medskip

\noindent {\bf Kinematic compatibility:} for all $t \in [0,T]$,
\begin{equation}\label{eq:kc2}
Eu(t)=e(t)+p(t) \text{ in }\Om, \quad p(t)=(w(t)-u(t))\odot \nu \HH^{n-1} \text{ on }\Gamma_D;
\end{equation}

\medskip

\noindent {\bf Stress constraint:} for all $t\in [0,T]$,
\begin{equation}
\label{eq:sigmaK2}
\sigma(t):=\mathbf C e(t) \in \K \quad \text{$\LL^n$-a.e. in }\Om;
\end{equation}

\medskip

\noindent {\bf Equilibrium equation:} for all $t \in (0,T)$,
\begin{equation}
\label{eq:motion2}
-\Div \sigma(t)=f(t) \quad \text{ $\LL^{n}$-a.e. in }\Om;
\end{equation}

\medskip

\noindent {\bf Neumann boundary condition:} for a.e. $t\in (0,T)$,
\begin{equation}
\label{eq:Neumann2}
\sigma(t)\nu=g(t)\quad  \HH^{n-1}\text{-a.e. on }\Gamma_N;
\end{equation}

\medskip

\noindent {\bf Flow rule:} for a.e. $t\in (0,T)$,
\begin{equation}\label{eq:fr2}
H(\dot p(t))=[\sigma(t)\cdot \dot p(t)] \quad \text{ in }\Om \cup \Gamma_D;
\end{equation}

\medskip

\noindent {\bf Initial condition:} 
\begin{equation}\label{eq:ic2}
(u(0),e(0),p(0))=(u_0, e_0, p_0).
\end{equation}

Moreover, the stress component satisfies the following estimate: 
for every open set $\omega\subset\subset \Omega$ there exists a constant $C_\omega>0$
such that 
\begin{equation}\label{sig:hreg}
\sup_{t\in[0,T]}\|\nabla \sigma(t)\|_{L^2(\omega)}\leq C_\omega.
\end{equation}
\end{theorem}

\begin{proof}
Let $\alpha_j\to 0^+$ and $\lambda_j \to +\infty$, and let $(\sigma_j,v_j)$ be the evolution given by Theorem~\ref{thm:psialpha} with the choice $\alpha:=\alpha_j$ and $\lambda=\lambda_j$. By applying the Ascoli-Arzel\`a Theorem we deduce from \eqref{est1-sN} that there exists $\sigma\in H^1(0,T;L^2(\Omega;\Mn))$ such that, up to subsequences, for every $t\in[0,T]$ 
$$\sigma_j(t)\wto\sigma(t) \quad \text{weakly in } L^2(\Omega;\Mn)$$
and
$$\sigma_j\wto\sigma \quad \text{weakly in } H^1(0,T;L^2(\Omega;\Mn)).$$
Moreover, by \eqref{prop-eq} we infer that $\sigma \in L^\infty(0,T;H^1_{loc}(\Om;\Mn))$ and \eqref{sig:hreg} holds. It thus follows that $\sigma(0)=\sigma_0$ and $-\Div\sigma(t)=\LL(t)$ in $[H^1_{\Gamma_D}(\Om;\R^n)]'$ for every $t\in[0,T]$. This last condition corresponds to the weak formulation of the equilibrium equation \eqref{eq:motion2} and the Neumann boundary condition \eqref{eq:Neumann2}.

Concerning the velocities, from \eqref{est1-vN} we get that up to a further subsequence
$$
v_j\wto v\qquad\text{weakly* in }L^2(0,T;BD(\Om)).
$$
We now set
$$
u(t):=u_0+\int_0^t v(s)\, ds \quad \text{ for every } t\in[0,T].
$$
We define $e(t):=\A\sigma(t)$ and $p(t):=Eu(t)-e(t)$ in $\Omega$, $p(t):=(w(t)-u(t))\odot \nu\HH^{n-1}$ on $\Gamma_D$.
It is thus clear that $(u(t),e(t),p(t))\in \Acal^{\rm qst}_{w(t)}$ from which kinematic compatibility \eqref{eq:kc2} and the initial condition \eqref{eq:ic2} follow.

The stress constraint \eqref{eq:sigmaK2} and the flow rule \eqref{eq:fr2} can be obtained by adapting Steps~4 and~5 in the proof of Theorem~\ref{thm:hreg2} for the dynamical case.
\end{proof}

\begin{remark}
The measure theoretic flow rule \eqref{eq:fr2} can be expressed in a pointwise way. Indeed, since 
$\sigma(t) \in H^1_{loc}(\Om;\R^n)$ for a.e. $t \in (0,T)$, there exists the (Borel measurable) quasi-continuous representative of $\sigma(t)$ for the $H^1$-capacity, which we denote by $\hat \sigma(t)$. Following \cite{DMDSM,FGM,BMo}, \eqref{eq:fr2} can be equivalently written as follows: for a.e. $t \in (0,T)$
$$H\left( \frac{\dot p(t)}{|\dot p(t)|} \right)= \hat \sigma(t) \cdot  \frac{\dot p(t)}{|\dot p(t)|} \quad |\dot p(t)|\text{-a.e. in }\Om.$$
\end{remark}

\section{Examples}\label{sec:ex}

In this section, we show that, in the dynamic case, although the solutions are (Sobolev) regular in the interior of the spatial domain, some jump might appear at the boundary, so that the Dirichlet boundary condition might fail to be satisfied.

We will consider two examples in the one-dimensional case where $\Omega=(0,L)$ with $L>0$. 
We will assume that $\C$ (and thus $\A$) is the identity tensor and $\K=[-1,1]\subset\R$, so that $H(\xi)=|\xi|$ for every $\xi\in\R$. As for notation, if $f:(0,T) \times (0,L) \to \R$ stands for a generic function, $f'$ denotes the partial derivative with respect to the space variable, while $\dot f$ denotes the partial derivative with respect to the time variable.

\subsection{The stationary case}

Let $w$ be a Dirichlet boundary data such that $w(0)=0$ and $w(L)=a$ (with $a \in \R$). Let $(u,\sigma,p) \in [W^{1,1}(0,L) \cap H^1_{\rm loc}(0,L)] \times [L^\infty(0,L)  \cap  H^1_{\rm loc}(0,L)] \times [\mathcal M([0,L]) \cap L^2_{\rm loc}(0,L)]$ be the unique solution of the stationary dynamic problem, i.e.,
$$
\begin{cases}
u'=\sigma+p & \text{ in } (0,L),\\
p(0)=u(0)-w(0), &\\
p(L)=w(L)-u(L),& \\
u-\sigma'=0 & \text{ in }(0,L),\\
|\sigma|\leq 1 & \text{ in } (0,L),\\
\sigma p=|p| & \text{ in } (0,L),\\
\sigma(0)(w(0)-u(0))=|w(0)-u(0)|,&\\
\sigma(L)(w(L)-u(L))=|w(L)-u(L)|.&
\end{cases}
$$
The first three conditions correspond to kinematic admissibility. The fourth equation is a stationary variant of the equation of motion, while the fifth condition is the stress constraint. The last three equations provide a stationary variant of the flow rule.

Note that the fourth equation actually shows that $\sigma'=u \in L^2(0,L)$, hence $\sigma \in H^1(0,L)$. In particular, $\sigma \in C^0([0,L])$.\medskip

{\bf Case 1.} If $|a| \leq \tanh(L)$, the (unique) solution turns out to be purely elastic, i.e., for all $x \in [0,L]$,
$$
\begin{cases}
\displaystyle u(x)=\frac{a\sinh(x)}{\sinh(L)},\\
\displaystyle\sigma(x)=u'(x)=\frac{a\cosh(x)}{\sinh(L)},\\
p\equiv 0.
\end{cases}
$$
Observe that the Dirichlet boundary condition $u=w$ on $\{0,L\}$ is satisfied in the usual sense, that is, no jumps appear at the boundary.\medskip

{\bf Case 2.} If $|a| > \tanh(L)$, we look for a solution of the form
$$
\begin{cases}
u(x)=2\alpha \sinh(x),\\
\sigma(x)=2\alpha\cosh(x),\\
p(x)=0
\end{cases}
$$
for all $x \in (0,L)$, for some some constant $\alpha \in \R$ to be specified later. The boundary condition $u(0)=0=w(0)$ is clearly satisfied. Note that the stress constraint $|\sigma|\leq 1$ imposes
$$|\alpha| \leq \frac{1}{2\cosh(L)}.$$
If $u(L)=w(L)$, then $\alpha=\frac{a}{2\sinh(L)}$, which enters into contradiction with $|a| > \tanh(L)$.  As a consequence, $u(L)\neq w(L)$ and the flow rule on the boundary implies that $|\sigma(L)|=1$, i.e.,
$$\alpha=\pm \frac{1}{2\cosh(L)},$$
and also 
$$
0\leq \sigma(L)(a-u(L))=\sigma(L)\big( a - 2\alpha \sinh(L)\Big) = \pm \big(a \mp \tanh(L)\big).
$$
Since we assume that $|a| >\tanh(L)$, we obtain that
$$\alpha= \frac{1}{2\cosh(L)} \quad \text{ if }a >\tanh(L),$$
while
$$\alpha=-\frac{1}{2\cosh(L)} \quad \text{ if }a <- \tanh(L).$$
In both these cases we thus get the existence and uniqueness of a solution that does not satisfy the Dirichlet boundary condition at $L$. The displacement $u$ experiences a jump at $x=L$ and the plastic strain is concentrated at that point, i.e., $p=(a-2\alpha\sinh(L)) \delta_L$.

\subsection{The evolutionary case}\label{subsec52}

The length $L>0$ of the space interval being fixed, let $a$ and $T$ be such that
$$0<a<\tanh(L)< a e^T.$$
Let $w$ be a Dirichlet boundary data such that $w(t,0)=0$ and $w(t,L)=a e^t$ for all $t \in [0,T]$. Let $(u,\sigma,p)$ be the unique solution of the evolutionary dynamic problem, i.e.,
\begin{equation}\label{eq:1Dpb}
\begin{cases}
u'=\sigma+p & \text{ in } (0,T) \times (0,L),\\
p(t,0)=u(t,0)-w(t,0) &\text{ for all }t \in [0,T],\\
p(t,L)=w(t,L)-u(t,L)& \text{ for all }t \in [0,T],\\
\ddot u-\sigma'=0 & \text{ in }(0,T) \times(0,L),\\
|\sigma|\leq 1 & \text{ in } (0,T) \times(0,L),\\
\sigma(t) \dot p(t)=|\dot p(t)| & \text{ in } (0,L),\quad \text{ for a.e. }t \in [0,T],\\
\sigma(t,0)(\dot w(t,0)- \dot u(t,0))=|\dot w(t,0)-\dot u(t,0)|, &\text{ for all }t \in [0,T],\\
\sigma(t,L)(\dot w(t,L)-\dot u(t,L))=|\dot w(t,L)-\dot u(t,L)| &\text{ for all }t \in [0,T].
\end{cases}
\end{equation}

Let $t_0\in (0,T)$ be such that
$$a e^{t_0}=\tanh(L).$$
On the time interval $[0,t_0]$ the solution coincides with the purely elastic solution, i.e.,
\begin{equation}\label{eq:0-t_0}
\begin{cases}
\displaystyle u(t,x)=\frac{ae^t\sinh(x)}{\sinh(L)},\\
\displaystyle\sigma(t,x)=u'(t,x)=\frac{ae^t \cosh(x)}{\sinh(L)},\\
p\equiv 0.
\end{cases}
\end{equation}

In order to compute the solution for the subsequent times $t>t_0$, we introduce the following sets:
$$A=\left\{(t,x) \in (t_0,T] \times [0,L] :\; x < \cosh^{-1}\left(\frac{\sinh(L)}{ae^t}\right)\right\},$$
$$B=\left\{(t,x) \in (t_0,T] \times [0,L] : \; x > \cosh^{-1}\left(\frac{\sinh(L)}{ae^t}\right)\right\},$$
and the curve
$$\Gamma=\left\{(t,x) \in (t_0,T] \times [0,L] :\; x =\cosh^{-1}\left(\frac{\sinh(L)}{ae^t}\right)=:\gamma(t) \right\},$$
which is the interface between $A$ and $B$ (see Fig.~\ref{figure}).

For all $(t,x) \in A$ we set
\begin{equation}\label{eq:t_0-A}
\begin{cases}
\displaystyle u(t,x)=\frac{ae^t\sinh(x)}{\sinh(L)},\\
\displaystyle\sigma(t,x)=u'(t,x)=\frac{ae^t \cosh(x)}{\sinh(L)}, \\
p\equiv 0.
\end{cases}
\end{equation}
Note that $|\sigma|<1$ in $A$ and that $A$ only contains the left Dirichlet boundary $(t_0,T] \times~\{0\}$.

We now compute the solution in $B$. We first set $\sigma\equiv 1$ in $B$. Note that this definition ensures that $\sigma$ is continuous across the curve $\Gamma$ and across the set $\{t=t_0\}$. We next use the equation of motion in $B$ to get that
$\ddot u=\sigma'=0$ in $B$, hence
$$u(t,x)=f(x) t + g(x) \quad \text{ for all }(t,x) \in B$$
for some functions $f$, $g$. 

On the other hand, since for all $t \in (t_0,T)$, the functions $x \mapsto u(t,x)$ and $x \mapsto \dot u(t,x)$ belong to $H^1_{\rm loc}(0,L)$, they must be continuous at the point $x=\gamma(t)$. Writing these conditions leads to
$$\tanh(x)=f(x) \ln \left( \frac{\sinh(L)}{a\cosh(x)}\right)+g(x) \quad \text{ and } \quad\tanh(x)=f(x),$$
or still
$$f(x)=\tanh(x), \quad g(x)=\tanh(x)\left[1-\ln\left(\frac{\sinh(L)}{a\cosh(x)}\right)\right].$$
For all $(t,x) \in B$ we thus set
\begin{equation}\label{eq:t_0-B}
\begin{cases}
\displaystyle u(t,x)=\tanh(x)\left[t+1-\ln\left(\frac{\sinh(L)}{a\cosh(x)}\right)\right],\\
\displaystyle\sigma(t,x)=1,
\end{cases}
\end{equation}
and
\begin{equation}\label{MG:exdp}
p(t,\cdot)=(u'(t,\cdot)-1)\LL^1+( w(t,L)-u(t,L))\delta_L
\end{equation}
in $B\cup\Gamma$. The continuity of $\sigma$, $u$, and $\dot u$ across $\Gamma$ and $\{t=t_0\}$ ensures that the equation of motion $\ddot u-\sigma'=0$ is satisfied in $(0,T) \times (0,L)$. By construction, the additive decomposition, as well as the stress constraint, are also satisfied. It remains to check the validity of the flow rule.

Since $p$ is continuous on $(0,T) \times (0,L)$, from its explicit expression we get that
$$
\dot p(t,x)=\left(1-\tanh^2(x)\right){\bf 1}_{{\rm int}B}(t,x) \geq 0
$$
on $(0,T) \times (0,L)$, hence $\sigma(t,x) \dot p(t,x)=|\dot p(t,x)|$ for all $(t,x) \in (0,T) \times (0,L)$. The flow rule is thus satisfied in the interior of the space-time region. 

For what concerns the flow rule on the boundary, we first notice that, according to \eqref{eq:0-t_0} and \eqref{eq:t_0-A}, we have $u(t,0)=0$, which also proves the validity of the flow rule at $x=0$. 

At the other boundary point $x=L$, we have $u(t,L)=a=w(t,L)$ for all $t \in [0,t_0]$ by \eqref{eq:0-t_0}. Moreover, for all $t\in (t_0,T]$ we have
$$w(t,L)-u(t,L)=ae^t -\tanh(L)\left[t+1-\ln\left(\frac{\tanh(L)}{a} \right)\right],$$
which is a strictly increasing function of $t$ over $(t_0,T]$. As a consequence, $\dot w(t,L)-\dot u(t,L)>0$ for all $t >t_0$, and using that $\sigma(t,L)=1$ by \eqref{eq:t_0-B} and the expression \eqref{MG:exdp}, we get the validity of the flow rule at $x=L$.

\begin{figure}
			\begin{tikzpicture}
%				\filldraw[color=black!5, fill=black!5] (0,0) -- (1,0) arc (0:90:1);
%				\filldraw[color=black!5, fill=black!5] (0,0) -- (-1,0) arc (180:270:1);
%				\filldraw[color=black!5, fill=black!5] (0,0) -- (0,2) arc (90:180:2);
%				\filldraw[color=black!5, fill=black!5] (0,0) -- (0,-2) arc (270:360:2);
%				\draw[thick] (1,0) arc [start angle=0, end angle=90, radius=1cm];
%				\draw[thick] (-1,0) arc [start angle=180, end angle=270, radius=1cm];
%				\draw[thick] (0,2) arc [start angle=90, end angle=180, radius=2cm];
%				\draw[thick] (0,-2) arc [start angle=270, end angle=360, radius=2cm];
%				
\draw[->]  (-0.5,0) -- (3.5,0);
\draw[->]  (0,-0.5) -- (0,4.5);
\node[below]  at (3.6,0) {\small$x$};
\node[left] at (0,4.6) {\small$t$};
\def\ya{1}
  \def\yb{4}
  \def\xa{{3-ln(cosh(3))}}
  \filldraw[fill=red!30]  (3,0) -- (3,4) -- (0,4) -- (0,0) -- (3,0);

\filldraw[domain=\ya:\yb, smooth, variable=\y, fill=blue!30] plot({3-ln(cosh(\y-1))}, \y);
\filldraw[blue!30] (0,0) -- (3,0) -- (3,1) -- (\xa,4) -- (0,4) -- (0,0);
\draw[dashed]  (0,\ya) -- (3,\ya);
\node [left] at (0,\ya) {\small$t_0$};
\node [left] at (0,4) {\small$T$};
\draw (0,0) -- (3,0) -- (3,1);
\draw  (\xa,4) -- (0,4) -- (0,0);
\node [below left] at (0,0) {\small$0$};
\node [below] at (3,0) {\small$L$};
%\node [left] at (1.5,2) {\small$A$};
\node [left] at (2.5, 3.5) {\small$B$};
\node [left] at (1.5, 3) {\small$\Gamma$};
%				\node [below] at (3,0) {\small$x_1$};
%				\draw[thick]  (1,0) -- (2,0);
%				\draw (-3,0) -- (-2,0);
%				\draw[thick] (-2,0) -- (-1,0);
%				\draw[dashed]  (-1,0) -- (1,0);
%				\draw[->]  (0,-0.5) -- (0,3);
%				\node [left] at (0,3) {\small$x_2$};
%				\draw[thick]  (0,1) -- (0,2);
%				\draw (0,-3) -- (0,-2);
%				\draw[thick] (0,-2) -- (0,-1);
%				\draw[dashed]  (0,-1) -- (0,1);
				%\node [above] at (1,-1) {\small$\Omega$};
			\end{tikzpicture}
			\caption{The elastic region (in light blue) and the plastic region (in light red) in the example of Section~\ref{subsec52}.}\label{figure}
		\end{figure}
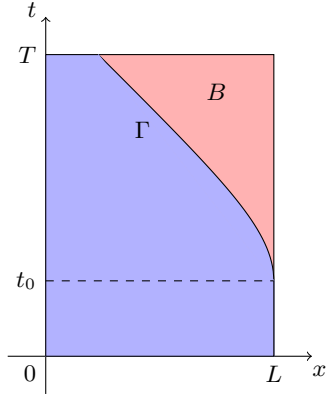

In conclusion, \eqref{eq:0-t_0}--\eqref{MG:exdp} is the unique solution to the one-dimensional evolutionary problem \eqref{eq:1Dpb}. The velocity $\dot u$  experiences a boundary jump at $x=L$ provided the final time $T$ is large enough (actually for all times $t \in (t_0,T]$). This shows the possibility to get spatial singularities at the boundary, although solutions are smooth in the interior of the space-time domain.

\section*{Appendix}

In this appendix we collect some results that were of use in the previous sections and show that the Hosford criterion, mentioned in the introduction, is covered by our analysis.

The lemma below, which was applied in the proof of Theorem~\ref{thm:hreg2}, follows by an easy adaptation of \cite[Lemma 2.35]{AFP}.

\begin{lemma}\label{lemma:AFP}
Let $\lambda$ be a positive Radon measure in an open set $U \subset\R^N$ and, for every $k \in\N$, let $f_k:U\to\R$ be a Borel function such that
$$
\int_U |f_k|\, d\lambda<+\infty.
$$
Then
$$
\int_U \left(\sup_k f_k\right)_+\,d\lambda =\sup\bigg\{ \sum_{k\in I}\int_{A_k} f_k\, d\lambda\bigg\}.
$$
where the supremum at the right-hand side is taken over all finite sets $I\subset\N$ and all families $\{A_k: k\in I\}$ of pairwise disjoint open sets with compact closure in $U$.
\end{lemma}

We now state and prove a result that was used in the proof of Theorem~\ref{thm:psialpha}.

\begin{lemma}\label{lm:CL}
Assume hypotheses $(H_1)$, $(H_2)$, $(H'_4)$, $(H'_5)$ and $(H_6)$. Let $\Psi:\Mn\to\Mn$ be a Lipschitz continuous function. 
Then the problem
\begin{equation}\label{NH-model}
\begin{cases}
\A\dot\sigma +\Psi(\sigma)=Ev & \text{ in }\Omega\times (0,T), 
\\
-\Div\sigma= f & \text{ in }\Omega\times (0,T),
\\
v=\dot w &\text{ on }\Gamma_D\times (0,T),
\\
\sigma\nu=g & \text{ on }\Gamma_N\times (0,T),
\\
\sigma(0)=\sigma_0
\end{cases}
\end{equation}
has a unique solution $(\sigma, v)\in H^1(0,T; H_{\Div}(\Omega;\Mn))\times L^2(0,T; H^1(\Omega;\R^n))$.
\end{lemma}

\begin{proof}
On $L^2(\Omega;\Mn)$ we consider the scalar product
\begin{equation}\label{scpr}
\langle\sigma,\tau\rangle_{\A}:=\int_\Om \A\sigma\cdot\tau\, dx \quad \text{for every }\sigma,\tau\in L^2(\Omega;\Mn),
\end{equation}
which is topologically equivalent to the standard scalar product by \eqref{coercA}. We introduce the set
$$
\Sigma_0(\Omega):=\{ \sigma\in H_{\Div}(\Omega;\Mn) :\ \Div\sigma=0 \text{ in } \Omega,  \;\;\sigma\cdot \nu=0 \text{ on }\Gamma_N\},
$$
which is a closed subspace of $L^2(\Omega;\Mn)$, and we denote the projection onto $\Sigma_0(\Omega)$ with respect to the scalar product \eqref{scpr} by $\Pi_0$.

We consider the following problem:  find $\theta\in H^1(0,T; \Sigma_0(\Omega))$ such that
\begin{equation}\label{Cauchy-pb-L}
\begin{cases}
 \Pi_0(\dot \rho(t)+\dot\theta(t))+ \Pi_0\big(\A^{-1}\Psi( \rho(t)+\theta(t))\big)=\Pi_0\big(\A^{-1}E\dot w(t)\big) & \text{ in }\Sigma_0(\Omega),
\\
\theta(0)=\sigma_0-\rho(0).
\end{cases}
\end{equation}
Let 
$$
\Lambda:[0,T]\times\Sigma_0(\Omega)\to\Sigma_0(\Omega)
$$
given by
$$
\Lambda(t,\vartheta):= \Pi_0(\dot \rho(t))+\Pi_0\big(\A^{-1}\Psi( \rho(t)+\vartheta)\big)-\Pi_0\big(\A^{-1}E\dot w(t)\big).
$$
Since $\Lambda(t,\cdot)$ is Lipschitz continuous for a.e.\ $t\in[0,T]$ and $\Lambda(\cdot, \vartheta)\in L^2(0,T;\Sigma_0(\Omega))$ for every $\vartheta\in\Sigma_0(\Omega)$, existence and uniqueness of solutions to problem \eqref{Cauchy-pb-L} follow from the Cauchy-Lipschitz Theorem. Now, the first equation in \eqref{Cauchy-pb-L} implies that
$$
\langle \dot \rho(t)+\dot\theta(t),\tau\rangle_{\A}
+\langle \A^{-1}\Psi( \rho(t)+\theta(t)), \tau\rangle_{\A}
= \langle \A^{-1} E\dot w(t), \tau\rangle_{\A}
$$
for every $\tau\in \Sigma_0(\Omega)$, that is,
$$
\int_\Om \left(\A( \dot \rho(t)+\dot\theta(t))+ \Psi( \rho(t)+\theta(t))- E\dot w(t)\right)\cdot\tau\, dx =0
$$
for every $\tau\in \Sigma_0(\Omega)$. This implies that for a.e.\ $t\in[0,T]$ there exists a unique
 $z(t)\in H^1(\Omega;\R^n)$ with $z(t)=0$ on $\Gamma_D$ such that
$$
\A( \dot \rho(t)+\dot\theta(t))+ \Psi( \rho(t)+\theta(t)) - E\dot w(t) 
= Ez(t).
$$
We set $\sigma(t):= \rho(t)+\theta(t)$ and $v(t):= z(t)+\dot w(t)$. We observe that $v\in L^2(0,T; H^1(\Omega;\R^n))$ by construction.
Thus, we have found a pair $(\sigma, v)$ satisfying \eqref{NH-model}.
\par
On the other hand, if $(\sigma, v)$ is a solution to \eqref{NH-model}, then $\theta(t):=\sigma(t)-\rho(t)$ satisfies \eqref{Cauchy-pb-L}
and is therefore uniquely determined. Uniqueness of $v$ then follows.
\end{proof}

We conclude this section by showing that the Hosford criterion fits in with the assumptions of our regularity results.

\begin{proposition}\label{prop:Hosford}
Let $p \geq 2$. Let $F: \MD \to [0,+\infty)$ be the function defined by
$$F(\sigma)=\sum_{1 \leq i < j  \leq n}|\sigma_i-\sigma_j|^p \quad \text{ for all } \sigma\in \MD,$$
where $\sigma_1,\ldots,\sigma_n$ are the eigenvalues of $\sigma$, and let
$$K=\big\{\sigma \in \MD : \quad F(\sigma) \leq 1\big\}.$$
Then $K$ is a compact and convex subset of $\MD$. Moreover, its boundary $\partial K$ is a $C^2$ hypersurface and its second fundamental form is positive definite at every point of $\partial K$.
\end{proposition}

\begin{proof}
Let us define the function $f:\R^n \to [0,+\infty)$ by
$$f(x_1,\ldots,x_n)=\sum_{1 \leq i < j \leq n} |x_i-x_j|^p \quad \text{ for all } (x_1,\ldots, x_n) \in \R^n.$$
Clearly $f$ is symmetric, i.e., for any permutation $\varphi:\mathbb N^n \to \mathbb N^n$ 
$$f(x_1,\ldots,x_n)=f(x_{\varphi(1)}, \ldots, x_{\varphi(n)})\quad \text{ for all }(x_1,\ldots,x_n) \in \R^n.$$
By definition of $F$, for every $\sigma\in \MD$ we have that $F(\sigma)=f(\sigma_1,\ldots,\sigma_n)$, where $\sigma_1,\ldots,\sigma_n$ are the eigenvalues of $\sigma$. Since $f \in C^2(\R^n)$,  \cite[Theorem 5.5]{Ball} shows that $F \in C^2(\MD)$.  
Moreover, by \cite[Appendix A]{HM} we get that 
\begin{equation}
\label{eq:Fconvex}
D^2F(\sigma)\xi \cdot \xi >0 \quad \text{ for every } \xi\in \MD,
\end{equation}
so that $F$ is strictly convex on $\MD$. We infer that $K$ is a closed and convex set and, since it is bounded, it is also compact. 

Let us turn now to the regularity properties of $\partial K$.
Since $F$ is $p$-homogeneous, it follows from the Euler formula that
$$p F(\sigma)=DF(\sigma) \cdot \sigma \quad \text{ for all } \sigma \in \MD.$$
As a consequence, since $F=1$ on $\partial K$, there is an open neighborhood $U$ of $\partial K$ such that $DF \neq 0$ in $U$. In particular, $\partial K=\{F=1\}$ is a co-dimension $1$ submanifold of $\MD$ of class $C^2$. 

Let us introduce the Gauss map $N:U \to \MD$ defined by
$$N(\sigma)=\frac{DF(\sigma)}{|DF(\sigma)|} \quad \text{ for all } \sigma \in U.$$
Then, $N$ is of class $C^1$ in $U$ with
$$DN(\sigma)=\frac{D^2F(\sigma)}{|DF(\sigma)|} - \frac{DF(\sigma) \otimes [D^2F(\sigma) DF(\sigma)]}{|DF(\sigma)|^3} \quad \text{ for all } \sigma \in U.$$
In particular, for all $\xi \in \MD$ we have
$$DN(\sigma)\xi\cdot \xi=\frac{D^2F(\sigma)\xi \cdot \xi}{|DF(\sigma)|} - \frac{DF(\sigma)\cdot \xi}{|DF(\sigma)|^3}\big(D^2F(\sigma) DF(\sigma)\cdot \xi\big).$$
Using that the tangent space $T_\sigma(\partial K)$ to $\partial K$ at $\sigma \in \partial K$ is given by $T_\sigma(\partial K)=DF(\sigma)^\perp$, we deduce that for all $\xi \in T_\sigma(\partial K)$ 
$$DN(\sigma)\xi\cdot \xi=\frac{D^2F(\sigma)\xi\cdot \xi}{|DF(\sigma)|}.$$
By \eqref{eq:Fconvex} we infer that that $DN(\sigma)\xi\cdot \xi>0$ for all $\sigma \in \partial K$ and all $\xi \in T_\sigma(\partial K)$. As a consequence, the second fundamental form $A(\sigma)$ of $\partial K$ is positive definite for all $\sigma \in \partial K$.
\end{proof}

\vskip20pt\noindent
\textbf{Acknowledgements.}
JFB is supported by a public grant from the Fondation Math\'ematique Jacques Hadamard. AG and MGM acknowledge support by PRIN 2022 n.2022J4FYNJ funded by MUR, Italy, and by the European Union -- Next Generation EU.  AG and MGM are members of GNAMPA--INdAM.

\end{document}